\documentclass{amsart}
\usepackage{amssymb,amsmath,amsthm,epsf,epsfig,dsfont,bbm}
\usepackage{latexsym}
\usepackage[utf8]{inputenc}
\usepackage{hyperref}
\hypersetup{
	colorlinks   = true,
	citecolor    = teal,
	linkcolor    = purple 
}
\usepackage{upref, eucal}
\usepackage[all]{xy}

\newcommand {\nc} {\newcommand}
\newcommand {\enm} {\ensuremath}

\def \d{\delta}
\nc {\nd}{\partial}

\nc {\bdm} {\begin{displaymath}}
\nc {\edm} {\end{displaymath}}

\newtheorem {theorem} {\bf{Theorem}}[section]
\newtheorem {lemma}[theorem] {\bf Lemma}
\newtheorem {proposition}[theorem] {\bf Proposition}

\newtheorem {corollary}[theorem] {\bf Corollary}
\numberwithin {equation}{section}

\newcommand\FF{\mathbb{F}}\newcommand\QQ{\mathbb{Q}}
\newcommand\ZZ{\mathbb{Z}}

\newcommand{\Ou}{\enm{\mathcal{O}}}

\newcommand{\A}{\enm{\mathbb{A}}}

\nc{\J}{\enm{\mathcal{J} }}
\nc {\Z} {\enm{\mathbb{Z}}}
\nc {\form}[1] {\enm{\mbox{\underline{for}}}_{#1}}
\nc {\prol}[1] {\enm{\mbox{\underline{prol}}_{{#1}^*}}}

\nc {\stk} {\stackrel}

\newcommand{\map}{\rightarrow}

\newcommand{\inj}{\hookrightarrow}

\newcommand{\dualmod}[1]{{#1}^{\vee}}

\newcommand{\Pn}[2] {\ensuremath{ {\mathbb{P}}^{#1}_{#2}}}
\nc{\Quot}[3]{\enm{ {\mathfrak{Quot}_{ {#1}/{#2}/{#3}}}}}
\nc{\Hilb}[2]{\enm{ {\mathfrak{Hilb}_{ {#1}/{#2}}}}}
\newcommand{\mfrak}[1]{\mathfrak{#1}}

\newcommand{\bb}[1]{\mathbb{#1}}
\newcommand{\mcal}[1]{\mathcal{#1}}

\newcommand{\Q}{\mathbb{Q}}
\nc {\Coh}[4] {\ensuremath{H^{#1}(\Pn{#2}{},{#3}({#4}))}}
\nc {\Ch}[3] {\enm{H^{#1}(X_t,{#2}_t({#3}))}}
\nc {\Qphi}[4]{\enm{ {\mathfrak{Quot}^{~#4}_{ {#1}/{#2}/{#3}}}}}
\nc {\Gra}[4]{\enm{ {\mathfrak{Grass}_{#2}({#3},{#4})}}}
\nc {\HomA}[2]{\enm{\mathrm{Hom}_A{#1}{#2}}}
\nc {\tr}{\mathrm{tr}}

\nc {\C}[2]{\enm{\left(\begin{array}{l} {#1} \\ {#2} \end{array} \right)}}
\nc {\mat}[4]{\enm{\left(\begin{array}{ll}{#1} & {#2} \\ {#3} & {#4}
\end{array}\right)}}

\def \vp{\varphi}
\def \mb{\mbox}

 \def \Z{{\mathbb Z}}

   \def \h{\hat{\ }}

\def \d{\delta}   \def \bF{{\bf F}}

  \def \bX{{\bf X}} \def \bH{{\bf H}}
   \def \bF{{\mathbb F}}

\def \hG{\hat{\mathbb{G}}_{\mathrm{a}}}

\def \R1{R((q))[q']\h}

\usepackage{xcolor}

\newcommand{\lam}{\lambda}
\DeclareMathOperator{\Spec}{\mathrm{Spec}}
\DeclareMathOperator{\Spf}{\mathrm{Spf}}
\DeclareMathOperator{\Lie}{\mathrm{Lie}}
\DeclareMathOperator{\rk}{\mathrm{rk}}
\newcommand{\Hom}{\mathrm{Hom}}
\newcommand{\End}{\mathrm{End}}
\newcommand{\Ext}{\mathrm{Ext}}

\newcommand{\bI}{{\bf I}}

\newcommand{\beqar}{\begin{eqnarray*}}
\newcommand{\eeqar}{\end{eqnarray*}}

\newcommand{\mff}{\mfrak{f}}

\nc{\bx}{\mathbf{x}}
\nc{\by}{\mathbf{y}}
\nc{\bz}{\mathbf{z}}
\nc{\ba}{\mathbf{a}}
\nc{\Fp}{\tilde{F}}
\nc{\Rp}{\tilde{R}}
\nc{\mlow}{m_{\mathrm{l}}}
\nc{\mup}{m_{\mathrm{u}}}
\nc{\ord}{\mb{ord }}
\nc{\bXp}{\bX_{\mathrm{prim}}}
\nc{\bPsi}{\mathbf{\Psi}}
\nc{\mult}{\mathrm{mult}}
\nc{\mbB}{\mathbbm{B}}
\nc{\mfor}[1]{{#1}^{\mathrm{for}}}
\nc{\Hdr}{\bH^1_{\mathrm{dR}}}
\nc{\Hcr}{\bH^1_{\mathrm{cris}}}
\nc{\Fc}{F_{\mathrm{cris}}}
\nc {\Hd}{\bH_{\d}}
\nc{\mn}{[m]n}

\nc{\tF}{\tilde{F}}
\nc{\fra}{\mfrak{f}}
\nc{\bt}{{\bf t}}
\nc{\Nn}{N^{\mn}}
\nc{\del}{\Delta}
\nc{\tilW}{\tilde{W}}
\nc{\bo}{{\bf b}}
\newcommand{\nHom}{\widetilde{\mathrm{Hom}}}
\nc{\Di}[2]{\Delta^{#1}i^*\d^{#2}}
\nc{\di}[1]{i^*\d^{#1}}

\newcommand{\cblue}{\color{black}}
\newcommand{\cblu}[1]{\cblue {#1} \color{black}}
\newcommand{\cbl}[1]{\color{black} {#1} \color{black}}
\newcommand{\mg}{g}
\newcommand{\mr}[1]{\mathrm{#1}}
\newcommand{\Iso}{FIso}
\newcommand{\vpi}{\mathrm{val}_\pi}

\title{Delta Characters and Crystalline Cohomology}
\author{Sudip Pandit and Arnab Saha}
\date{}
\email{sudip.pandit@iitgn.ac.in, arnab.saha@iitgn.ac.in}
\address{Indian Institute of Technology Gandhinagar, Gujarat 382355, India}


\begin{document}
\begin{abstract}

The first part of the paper develops the theory of $m$-shifted $\pi$-typical 
Witt vectors which can be viewed as subobjects of the usual $\pi$-typical 
Witt vectors. We show that the shifted Witt vectors admit a delta structure 
that satisfy a canonical identity with the delta structure of the usual
$\pi$-typical Witt vectors. Using this theory, we prove that the generalized
kernels of arithmetic jet spaces are jet spaces of the kernel at the 
first level. This also allows us to interpret the arithmetic Picard-Fuchs 
operator geometrically.

For a $\pi$-formal group scheme $G$, by a previous construction,
one attaches a canonical filtered isocrystal $\mathbf{H}_\delta(G)$ associated
to the arithmetic jet spaces of $G$. In the second half of our paper, we 
show that $\mathbf{H}_\delta(A)$ is of finite rank if $A$ is an abelian scheme.
We also prove a strengthened version of a result of Buium on delta characters
on abelian schemes.
As an application, for an elliptic curve $A$ defined over $\mathbb{Z}_p$, we
show that our canonical filtered isocrystal $\mathbf{H}_\delta(A) \otimes
\mathbb{Q}_p$ is weakly admissible. In particular, if $A$ does not admit a 
lift of Frobenius,
we show  that $\mathbf{H}_\delta(A) \otimes \mathbb{Q}_p$ is 
isomorphic to the first crystalline cohomology $\mathbf{H}^1_{\mathrm{cris}}(A)
\otimes \mathbb{Q}_p$ in the category of filtered isocrystals. On the 
other hand, if $A$ admits a lift of Frobenius, then $\mathbf{H}_\delta(A)
\otimes \mathbb{Q}_p$ is isomorphic to the sub-isocrystal $H^0(A,\Omega_A)
\otimes \mathbb{Q}_p$ of $\mathbf{H}^1_{\mathrm{cris}}(A) \otimes 
\mathbb{Q}_p$. 

The above result can be viewed as a character theoretic interpretation of
the crystalline cohomology. 
The difference between the integral structures of 
$\mathbf{H}_\delta(A)$
and $\mathbf{H}_{\mathrm{cris}}^1(A)$ is measured by a delta modular form $f^1$
constructed by Buium.




\end{abstract}
\maketitle
\section{Introduction}

Let us fix a Dedekind domain $\Ou$ with finite residue fields 
and a nonzero prime ideal $\mfrak{p}$ in 
it. Let $k$ be the residue field at $\mfrak{p}$ with cardinality $q$ which is 
a power of
a prime $p$ and $\pi$ be a uniformizer of $\mfrak{p}\Ou_{\mfrak{p}}$. 
Let $R$ be an $\Ou$-algebra with a $\pi$-derivation $\d$ on it. 
Consider $X$ to be any scheme defined over $\Spec R$. In analogy with 
differential algebra, for all $n$, one defines the $n$-th arithmetic
jet functor as 
$$
J^nX(B):= X(W_n(B))
$$
where $W_n(B)$ is the $\pi$-typical Witt vectors of length $n+1$ 
for any $R$-algebra $B$ \cite{bor11a, drin76, Lars, joyal}. 
By \cite{bps,bor11b}, the functor $J^nX$ is representable
by an $R$-scheme, which we will continue to denote as $J^nX$. In the category of
$\pi$-formal schemes, $J^nX$ is precisely the arithmetic jet space
constructed by Buium \cite{bui95}. 

Some of the applications of our main
theorem will be in the category of $\pi$-formal schemes.
If $G$ is a $\pi$-formal smooth group scheme defined over $\Spf R$,
then the natural projection map
$u:J^nG \map G$ is a surjection of group schemes \cite{bor11b, bui00}. 
Let us denote the kernel
of $u$ as $N^nG$. Then they satisfy the following canonical short exact 
sequence of $\pi$-formal group schemes
\begin{align}\label{short1}
0 \map N^nG \map J^nG \stk{u}{\map} G \map 0.
\end{align} 
In the case when $A$ is an elliptic curve and $R$ is a $\pi$-adically 
complete discrete valuation ring of characteristic 0 whose
ramification is bounded by $p-2$, Buium shows that $N^1A \simeq \hG$
and therefore one obtains a canonical extension of the elliptic curve $A$ by 
$\hG$
\begin{align}
\label{short2}
0 \map \hG \map J^1A \map A \map 0.
\end{align}
Hence $J^1A$ gives rise to a canonical extension class $\eta_{\tiny{J^1A}} \in 
\Ext(A,\hG) \simeq H^1(A,\Ou_A)$. This class is trivial if and only if $A$ has
a canonical lift of Frobenius (which we will denote as CL). This leads to a 
remarkable new $\d$-modular function $f^1$ defined by Buium in \cite{bui00} as
$$
f^1(A,\omega) = \langle \omega, \eta_{\tiny{J^1A}} \rangle,
$$
where $\omega$ is an invertible $1$-form on the elliptic curve $A$
and $\langle ~, ~ \rangle$ is the pairing arising from Serre duality.
Note that as a (delta) modular function, $f^1$ has the property that it vanishes
whenever $A$ has a canonical lift of Frobenius. This is 
equivalent to having the Serre-Tate parameter $q(A)=1$. 
Using the above property of $f^1$, 
Buium and Poonen in \cite{bui09} show that the intersection of the Heegner 
points with any finite rank subgroup of a modular elliptic curve is finite.

The theory of $\d$-geometry and $\d$-modular forms is developed in a 
series of articles
such as 
\cite{barc, BL, bui00,bui-book,busa1,busa2,hurl}. In \cite{bui96}, Buium proved 
an effective Manin-Mumford conjecture using $\d$-geometry. In \cite{BS_b}, 
Borger and Saha construct canonical filtered  isocrystals associated to 
delta characters of a group scheme. We will use this construction in this 
article and prove comparison results with the crystalline cohomology for 
elliptic curves defined over $\Z_p$.
The equal characteristic analogue of the above construction
was done in \cite{BS_a, PS-1}. 

In the first part of this paper, we construct  $m$-shifted $\pi$-typical
Witt vectors $W_{\mn}(B)$ of length $m+n+1$ for any $\Ou$-algebra $B$. 
In \cite{BS_b}, Borger and Saha introduced $0$-shifted Witt vectors.
The shifted Witt vectors should be thought of as certain subrings of the usual
$\pi$-typical Witt vectors. Interstingly we show that such shifted 
Witt vectors admit a different $\d$-structure than the usual ones. The lift
of Frobenius associated to this $\d$-structure, called the {\it Lateral 
Frobenius} $\tilde{F}:W_{\mn}(B) \map
W_{[m]n-1}(B)$  satisfies the following canonical identity 
$$
F^{m+2} \circ I = F^{m+1} \circ I \circ \tilde{F}
$$ 
where $F:W_n(B) \map W_{n-1}(B)$ is the usual Frobenius map of Witt vectors 
and $I: W_{\mn}(B) \map W_{m+n}(B)$ is a natural map between rings.

Let $(X,P_0)$ denote a scheme $X$ over $\Spec R$ with a marked $R$-point
$P_0: \Spec R \map X$. Then composing with the map induced by $\exp_\d$
(analogue of Hasse-Schmidt differentiation map as in \cite{bps}, Proposition 
$2.10$), 
$P_0$ induces an $R$-point $P_m: \Spec R \map J^mX$. Consider the fiber 
product 
$$
N^{\mn}X := J^{m+n}X \times_{J^mX,P_m} \Spec R
$$
Here we would like to remark that the natural lift of 
Frobenius morphism $\phi_X: J^{m+n}X \map J^{m+n-1}X$ in general does not 
restrict to a morphism from $N^{\mn}X$ to $N^{[m]n-1}X$.
However using the lateral Frobenius on $m$-shifted Witt vectors, we show that 
the system of schemes $\{N^{\mn}X\}_{n=1}^\infty$ naturally become a 
prolongation sequence (for definition see Section \ref{pre}) 
of $S$-schemes when $X$ is affine. 

\cblue
Let us first explain our results in the setting of $R$-algebras.
Hence for an affine scheme $X = \Spec B$, the lateral Frobenius $\fra$ in 
Theorem \ref{lateral} induces a lift of Frobenius $\fra: N_{[m]n-1}B 
\map N_{[m]n}B$ for all $n \geq 1$. We will show in (\ref{Delta-desc}), 
that for all $n \geq 2$, $\fra$ naturally induces a unique $\pi$-derivation 
$\Delta : N_{[m]n-1}B \map N_{[m]n}B$ which satisfies
$$
\fra (a) = a^q + \pi \Delta (a),
$$
for all $a \in N_{[m]n-1}B$. This naturally makes the system of 
$R$-algebras 
$N_{[m]*}B = \{N_{[m]n}B\}_{n=1}^\infty$ into a prolongation sequence of 
$R$-algebras.

Let $R_* = R \stk{\d}{\rightarrow} R \stk{\d}{\rightarrow} \dots $ be the 
prolongation sequence
with the $\pi$-derivations at all levels to be the fixed $\pi$-derivation $\d$ 
on $R$.
By the universal property, as in Proposition $1.1$ in \cite{bui00} of 
canonical prolongation sequences of $R$-algebras, we have
\begin{align}
\label{upprol}
\Hom_R(N_{[m]1}B, N_{[m]1}B) \simeq \Hom_{R_*}(J_*(N_{[m]1}B),N_{[m]*}B). 
\end{align}
where $J_*(N_{[m]1}B)$ is the canonical prolongation sequence of $R$-algebras
as in Section $(1.2)$ of \cite{bui97}, with canonical $\pi$-derivations 
$\d: J_n(N_{[m]1}B) \map J_{n+1}(N_{[m]1}B)$ for all $n$.
Hence the identity map $g_0:= \mathbbm{1} \in 
\Hom_R(N_{[m]1}B,N_{[m]1}B)$ induces the following map of
prolongation sequences $\mg_* : J_*(N_{[m[1]}B) \map N_{[m]*}B$
 which is a system of $R$-algebra homomorphisms given by the following 
diagram where the $\pi$-derivations and the $R$-algebra maps commute at 
every level:
\begin{align}
\label{mgd3}
\xymatrix{
& \\
J_n(N_{[m]1}B) \ar[r]^{\mg_n} \ar@{.>}[u]& N_{[m]n+1}B \ar@{.>}[u]\\
J_{n-1}(N_{[m]1}B) \ar[r]^{\mg_{n-1}} \ar[u]^\d & N_{[m]n}B 
\ar[u]_\Delta \\
J_{1}(N_{[m]1}B) \ar[r]^{\mg_1} \ar@{.>}[u]^\d & N_{[m]2}B  
\ar@{.>}[u]_\Delta \\
N_{[m]1}B \ar[r]^{\mg_0 = \mathbbm{1}} \ar[u]^\d & N_{[m]1}B. 
\ar[u]_\Delta \\
} 
\end{align}

\begin{theorem}
\label{Niso1}
For all $m \geq 0$ the canonical map $\mg_*$ induces an isomorphism
$$
J_*(N_{[m]1}B) \simeq N_{[m]*}B 
$$
of prolongation sequences over $R_*$. In particular for all $n \geq 1$, the
above isomorphism implies 
$$
\mg_{n-1}: J_{n-1}(N_{[m]1}B)  \simeq N_{[m]n}B
$$
as $R$-algebras.
\end{theorem}

The above result is proved in Section \ref{kj}.

\color{black}

Note that, if $X$ is a group scheme over $\Spec R$ with $P_0$ as the 
identity section, then for all $m$, $J^mX$ is naturally a group scheme with
identity section $P_m$ as defined above. Then $N^{\mn}X$ is the kernel of 
the natural projection map of group schemes $u: J^{m+n}X \map J^mX$, that is,
we have the following exact sequence of group schemes
$$
0 \map  N^{\mn}X \map J^{m+n}X \stk{u}{\map} J^mX 
$$
for all $m$ and $n$. Observe that in the category of smooth $\pi$-formal
group schemes, the above  exact sequence is also right exact by Corollary $1.5$
of \cite{bui95}.

\cblue
The following result is a direct consequence of Theorem \ref{Niso1}.

\begin{theorem}
\label{affjet1}
If $X$ an affine scheme over $S$, then for all $m \geq 0$ we have 
$$
N^{[m]*}X \simeq J^*(N^{[m]1}X)
$$
as a canonical isomorphism of prolongation sequences of schemes over $S^*$.

In particular, for all $n \geq 1$, the above induces
a canonical isomorphism
$$
N^{\mn}X \simeq J^{n-1}(N^{[m]1}X)
$$
of schemes over $S$.
\end{theorem}

\color{black}

Now we will apply our result in the case of $\pi$-formal schemes. Assume that
$R$ is a $\pi$-adically complete discrete valuation ring and $l=R/\pi R$ be
its residue field. Let $X$ be a
$\pi$-formal scheme over $S = \Spf R$ with a marked point $P_0: \Spf R \map
X$. Then as a consequence of Theorem \ref{affjet1}, our next result 
characterizes 
$N^{\mn}X$ for any $\pi$-formal scheme $X$.

\cblue

\begin{theorem}
\label{formalcase1}
If $X$ is a $\pi$-formal scheme over $S$, then for all $m \geq 0$,
we have 
$$
N^{[m]*}X \simeq J^*(N^{[m]1}X)
$$
as a canonical isomorphism of prolongation sequences of $\pi$-formal 
schemes over $S^*$.

In particular, for all $n \geq 1$, the above induces a canonical
isomorphism
$$
N^{\mn}X \simeq J^{n-1}(N^{[m]1}X)
$$
of $\pi$-formal schemes over $S$.
\end{theorem}

\color{black}

Previously, the structure of the group scheme $N^nG(=N^{[0]n}G)$ in (\ref{short1}) was
unknown. In the case when $G$ is a smooth commutative $\pi$-formal group 
scheme, Buium in \cite{bui95} showed that $N^nG$ is a successive extension
of the additive group scheme $\hG$.
Let $\bb{W}_n$ denote the $\pi$-formal scheme ${\hat{\bb{A}}}^{n+1}$ endowed
with the group law of addition of Witt vectors. 
Then as an application 
of our Theorem \ref{formalcase1}, the canonical short exact sequence of 
(\ref{short1}) results in the following:


\cbl{
\begin{theorem}
\label{JetGg}
Let $R$ be of characteristic 0 satisfying $\vpi(p) \leq p-2$.
Let $G$ be a smooth commutative  $\pi$-formal group scheme of relative 
dimension 
$d$ over $\Spf R$. Then $N^{\mn}G \simeq \left(\bb{W}_{n-1}\right)^d$ as 
$\pi$-formal group schemes.
In particular we have the following short exact sequence of $\pi$-formal 
group schemes
$$
0 \map \left(\bb{W}_{n-1}\right)^d \map J^{m+n}G \map J^mG \map 0.
$$
\end{theorem}
}

We now discuss a consequence of our above results.
Let $L$ be any perfect field extension of $k$ and consider the 
perfect $\Ou$-algebra $B= W_\infty(L)$ and $B_n := W_n(L)\simeq B/\pi^{n+1}B$. 
Then consider the canonical exact sequence of groups 
\begin{align}
0 \map G(\pi^{n+1}B) \map G(B_n) \stk{u}{\map} G(L) 
\nonumber
\end{align}
where $u$ is the map induced from the quotient map of rings $B_n \map L$ and 
$G(\pi^{n+1}B)$ is the kernel of $u$. Note that $G(B_n) = G(W_n(L)) \simeq
J^nG(L)$. 
Then our Theorem \ref{formalcase1} implies that $N^nG \simeq J^{n-1}(N^1G)$.
This gives a geometric characterization of the group $G(\pi^{n+1}B)$ as follows
\begin{align}
G(\pi^{n+1}B) \simeq J^{n-1}(N^1G)(L) \simeq N^1G(B_{n-1}).
\nonumber
\end{align}
More so, if $R$ is of characteristic 0 such that 
$\vpi(p) \leq p-2$ and $G$ is a commutative smooth group scheme of 
dimension $d$, then our Theorem \ref{JetGg} implies $G(\pi^{n+1}B) \simeq
B_n^d$ where the group law is the one induced from the 
additive structure of the ring $B_n$. 
As for example, if $L= \bF_p$, then $B_n= \Z/p^{n+1}\Z$. Hence for any 
commutative smooth $\pi$-formal group scheme $G$ of relative dimension $d$
over $\Spf \Z_p$, we have 
$$
G(p^{n+1}\Z_p) \simeq \left( \Z/p^{n+1}\Z \right)^d
$$
as groups.

Given a $\pi$-formal group scheme $G$ over $\Spf R$, our Theorem 
\ref{formalcase1}
in the case $m=0$, gives for each $n$ the canonical lift of Frobenius 
$\fra: N^nG\simeq J^{n-1}(N^1G) \map N^{n-1}G \simeq J^{n-2}(N^1G)$ 
which satisfies 
\begin{align}
\label{latfrobcom}
	\phi \circ i \circ \fra = (\phi)^{\circ 2} \circ i.
\end{align}	
Consider the $R$-module $\varinjlim \Hom(N^nG,\hG)$. Then the pull-back map
$\fra^*$ makes the above into an $R\{\fra^*\}$-module. Then
we define as in \cite{BS_b},
$$
\Hd(G):= \varinjlim_n \frac{\Hom(N^nG,\hG)}{i^* \phi^*(\bX_{n-1}(G)_\phi)}.
$$
By (\ref{latfrobcom}), $\fra$ on $\varinjlim \Hom (N^nG, \hG)$ 
descends on $\Hd(G)$ and the resulting semilinear action on $\Hd(G)$ 
will be denoted as $\fra^*$. Now consider the $R$-module
$$
\bXp(G) := \varinjlim~ (\bX_n(G)/\phi^*(\bX_{n-1}(G)_\phi)).
$$
The above $R$-modules satisfy 
\begin{align}
\label{firstexact}
	0 \map \bXp(G) \map \Hd(G) \map \bI(G) \map 0
\end{align}
where $\bI(G)$ is an $R$-submodule of $\Ext (G,\hG)$ defined as in 
(\ref{bigI}). In the case when $A$ is an elliptic curve, $\bI(A)$ is the 
$R$-submodule generated by the class $\eta_{J^1A}$ associated to (\ref{short2}).
Then from Theorem $9.4$ in \cite{BS_b}, it follows that $\bI(A) = R\langle 
\eta_{J^1A}
\rangle \subseteq H^1(A,\Ou_A)$. The above inclusion is an equality 
if and only if $\eta_{J^1A}$ is a basis for $H^1(A,\Ou_A)$, which means that
$J^1A \simeq A^\sharp$ where $A^\sharp$ is the universal vectorial extension of 
$A$. This is equivalent to $f^1(A,\omega) \in R^*$ for all invertible 
$1$-forms $\omega$ on $A$.

Let $A$ a $\pi$-formal abelian scheme over $\Spf R$. Then we have the 
following map between the short exact sequences of $R$-modules as in
$(1.1)$ of \cite{BS_b}
\begin{align}
\xymatrix{
	0 \ar[r] & \bXp(A) \ar[d]^\Upsilon \ar[r] & \Hd(A) \ar[d]^\Phi \ar[r] & \bI(A) \ar@{^{(}->}[d] \ar[r] & 0 \\
	0 \ar[r] & H^0(A,\Omega_A) \ar[r] &\bH^1_{\mathrm{dR}}(A) \ar[r] & H^1(A,\Ou_A) \ar[r] & 0.
	}
\end{align}	
Apriori, it is not obvious as to whether $\bXp(A)$ and $\Hd(A)$ are finite
 free $R$-modules. However it was shown in \cite{BS_b} that the above 
$R$-modules when tensored with $K$ (the fraction field of $R$) are finite 
dimensional vector spaces over $K$. Our next result shows that the $R$-modules 
$\bXp(A)$ and $\Hd(A)$ are finite free over $R$.

\begin{theorem}\label{yeah}
Let $A$ be a $\pi$-formal abelian scheme of relative dimension $g$ 
over $\Spf R$. Then
we have 
\begin{enumerate}
\item The $R$-module $\bXp(A)$ is free of rank $g$.
\item The $R$-module $\Hd(A)$ is free satisfying
	$g \leq \rk_R\Hd(A) \leq 2g$.
\end{enumerate}		
\end{theorem}	

As a consequence of the above, we obtain the following (in 
Section \ref{intmod})
\begin{theorem}\label{fin-gen}
For an abelian scheme $A$ of relative dimension $g$ over $\Spf R$, 
$\bX_{\infty}(A)$  is  freely generated as an $R\{\phi^{*}\}$-module 
by $g$ $\d$-characters of order at most $g+1$.
\end{theorem}

Buium in \cite{bui95} proved the above result for $\bX_\infty(A)_K:=
\bX_{\infty}(A) \otimes_R K$. 
The positive characteristic analogue of the above results in the case of
Anderson modules were shown in \cite{PS-1}.

Let us now assume $R = W(l)$ and $K$ be its fraction field.
Let $$\mathrm{\Iso}(\Hd(A)):= (\Hd(A)_K,\mff^{*}, \Hd(A)_K^{\bullet})$$ 
denote the object in the category
of filtered isocrystals where $\mff^{*}$ is the semilinear operator on 
$\Hd(A)_K$ and $\Hd(A)_K^{\bullet}$ denote the filtration given by 
$\Hd(A)_K \supset \bXp(A)_K \supset \{0\}$. 

Let $A_0$ denote the special fiber of $A$ over $\Spec l$. Let $\Hcr(A)$
denote the first crystalline cohomology of $A_0$ with coefficients in $R$.
Then by the comparison theorem in \cite{BO}, we have $\Hcr(A) \simeq \Hdr(A)$.
Consider the filtered isocrystal
$$\mathrm{\Iso} (\Hcr(A)_K):= (\Hcr(A)_K, \Fc, 
\Hcr(A)_K^{\bullet})$$ where 
$\Fc $ is the semilinear operator on $\Hcr(A_0,W(l))_K$ induced
from the absolute Frobenius on $A_0$
and $\Hcr(A))^{\bullet}_K$ is the Hodge filtration given by 
$\Hcr(A)_K \supset H^0(A, \Omega_A)_K \supset \{0 \}$.
In the next result, our delta geometric object $\Hd(A)$ gives a 
character theoretic interpretation of the crystalline cohomology.

\begin{theorem}\label{Iso_crys} 
Let $A$ be an elliptic curve over $\Z_p$. Then $\mathrm{\Iso}(\Hd(A)_{\Q_p})$ is
a weakly admissible object in the category of filtered isocrystals.

(1) If $A$ is a non-CL elliptic curve then
$$
\mathrm{\Iso} (\Hd(A)_{\Q_p}) \simeq \mathrm{\Iso} (\Hcr(A)_{\Q_p})
$$
in the category of filtered isocrystals.

(2) If $A$ has CL then $$\mathrm{\Iso} (\Hd(A)_{\Q_p}) \simeq 
\mathrm{\Iso} (H^0(A, \Omega_A)_{\Q_p})$$ 
in the category of filtered isocrystals 
where $\mathrm{\Iso} (H^0(A,\Omega_A)_{\Q_p})$
is the one dimensional sub-object of $\mathrm{\Iso}(\Hcr(A)_{\Q_p})$.
\end{theorem}

Here we would like to remark on the comparison between the integral structures
of the $R$-modules $\Hd(A)$ and $\Hcr(A)$. The proof of Theorem 
\ref{Iso_crys} implies that the canonical $R$-module map $\Phi: \Hd(A) \map 
\Hcr(A)$ is injective. Hence $\Phi$ is an isomorphism if and only if 
$\bI(A) =  H^1(A,\Ou_A)$.
In other words $\Phi: \Hd(A) \simeq \Hcr(A)$ if and only if the delta
modular form $f^1(A,\omega)$ 
is an invertible element of $R$ for all invertible $1$-forms $\omega$ on $A$.
Hence the difference between the integral structures of $\Hd(A)$ and 
$\Hcr (A)$ is measured by the canonical class $\eta_{J^1A} \in H^1(A,\Ou_A)
$ which determines the value of the delta modular form $f^1$ at 
$(A,\omega)$.

Also in the light of the above results, one may ask for the true connection 
between $\Hd(A)$ and the prismatic cohomology by Bhatt and Scholze 
\cite{bhsch}.

\section{Plan of the paper} In Section \ref{pre}, we review the basic definitions of $\d$-rings, prolongation of schemes over a $\d$-ring, and arithmetic jet spaces. In Section \ref{swv}, we introduce the notion of shifted Witt vectors. 
Next, we define the centered polynomials in Section \ref{cp} and study their basic properties in Lemma \ref{sideal} and Proposition \ref{coord3}. 
In Section \ref{kj}, we prove that the generalized $n$-th kernels are, in fact, the $(n-1)$-th jet spaces of the first kernel. 
In Section \ref{intmod}, we briefly review the construction of $R$-modules 
$\bXp(A)$ and $\Hd(A)$ from \cite{BS_b,bui95}. 
We then prove the finiteness of the $R$-modules $\bXp(A)$ and $\Hd(A)$ 
when $A$ is an abelian scheme.

In Section \ref{last}, we interpret the arithmetic Picard Fuchs operator
associated to a delta character $\Theta$. Using this we prove Theorem 
\ref{Iso_crys} on the comparison of $\Hd(A)$ with the crystalline 
cohomology $\Hcr(A)$ where $A$ is an elliptic curve over $\Z_p$.

In Appendix \ref{FuncPoin}, an alternate proof of Theorem of \ref{Niso1} is
given using the functor of points approach. This proof has been pointed out
by the anonymous referee.

\section{Notation}
\label{notation}
We collect here some notations fixed throughout the paper.
\begin{align*}
	p &= \text{a prime number} \\
%
\mcal{O}&= \text{a Dedekind domain}\\
	\mfrak{p}&= \text{a fixed prime ideal of}~ \mcal{O}\\
	\pi &= \text{a generator of }\mfrak{p} \mathcal{O}_{\mfrak{p}} \\
	k &= \text{the residue field of $\mcal{O}$ at $\pi$ with cardinality $q$} \\
	R &= \text{a fixed $\mcal{O}$-algebra} \\
	\phi &= \text{an endomorphism of $R$ satisfying $\phi(x) \equiv x^q \bmod \mfrak{p}$, for all $x \in R$} \\
%
	S &= \Spec R \\
& \text{When $R$ is a $\pi$-adically complete discrete valuation ring then} \\
S &= \Spf R\\
	M_K &= K\otimes_R M, \text{ for any $R$-module $M$ and } 
		K=\mathrm{Frac}(R) \\
	\mfrak{m} &= \text{the maximal ideal of }R\\
	\vpi &= \text{the valuation on $R$ normalized such that }\vpi(\pi)=1\\
	e &= \text{the absolute ramification index $\vpi(p)\leq p-2$} \\
	l &= \text{the residue field of $R$} \\
G &= \text{a commutative smooth $\pi$-formal group scheme over $\Spf R$}\\
A &= \text{a $\pi$-formal abelian scheme over $\Spf R$}\\
a_{p}&= p+1-\#(A(\FF_{p})) \text{ when $A/\ZZ_{p}$ is an elliptic curve}\\
	\omega&= \text{normalized invariant differential 1-form of an elliptic curve.}
\end{align*}

\section{Preliminaries}\label{pre}

Let $\Ou$ be a Dedekind domain and $\mfrak{p}$ a non-zero prime ideal with
$k$ as the residue field and $q$ be the cardinality of $k$ where $q$ is a
power of a prime $p$. Let $\pi$ be one of the uniformizers of $\mfrak{p}$.
For any $\Ou$-algebra $B$ and $B$-algebra $A$, we define a $\pi$-derivation 
$\d$ as a set-theoretic theoretic map $\d:B \map A$ that satisfies  for 
all $x,y \in B$,

\cblue
$(i)~ \d (1) = 0$
\color{black}

$(ii)~ \d (x+y) = \d x + \d y + C_\pi(u(x),u(y)) $

$(iii)~ \d(xy) = u(x)^q \d y + u(y)^q \d x + \pi \d x \d y$

where $u:B \map A$ is the structure map and 
$$C_\pi(X,Y) = 
\left\{
\begin{array}{ll} \frac{X^q+Y^q -(X+Y)^q}{\pi}, & \mb{if }\mb{char } \Ou =0\\
0, & \mb{otherwise.}
\end{array}
\right.
$$ 
Given such a $\pi$-derivation $\d$, define $\phi(x):= u(x)^q + \pi\d x$
which is then a ring homomorphism satisfying
$$
\phi(x) \equiv u(x)^q \bmod \mfrak{p}.
$$
We will call such a $\phi$ a {\it lift of Frobenius} with respect to $u$. 
Fix an $\Ou$-algebra $R$ with a $\pi$-derivation $\d$ and call $S= \Spec R$.
Let $X$ and $Y$ be schemes over $S$. We say a pair $(u,\d)$ is a {\it 
prolongation} and we write $Y\stk{(u,\d)}{\map} X$, if $u:Y \map X$ is a 
map of $S$-schemes and $\d: \Ou_X \map u_*\Ou_Y$ is a $\pi$-derivation making
the following diagram commute:
\begin{align}
\label{proldia}
\xymatrix{
R \ar[r] & u_*\Ou_Y \\
R \ar[u]^\d \ar[r] &\Ou_X \ar[u]_\d
}
\end{align}

\cblu{In particular, when $Y = \Spec B$ and $X = \Spec A$ are affine schemes
where $B$ and $A$ are $R$-algebras and let (with slight abuse of notation) 
the induced map on the coordinate rings still be denoted by $u$. Then $\d:
A \map B$ is a $\pi$-derivation with respect to $u$ that respects the fixed
$\pi$-derivation on $R$ as given in the above diagram \eqref{proldia}.
}

We now recall the notion of prolongation and arithmetic jet spaces over $\pi$-formal schemes. Some of the definitions below were introduced by Buium
in the case of $\pi$-formal schemes. However their extension to the case 
of schemes is natural and without any technical challenges.
For a more detailed treatment of this section we refer to \cite{bui95, BS_b}.
As in page 103 in \cite{bui00}, a {\it prolongation sequence of $S$-schemes} 
is a sequence 
$$
S \stk{(u,\d)}{\longleftarrow} T^0 \stk{(u,\d)}{\longleftarrow} 
T^1 \stk{(u,\d)}{\longleftarrow} \cdots,
$$
where $T^i \stk{(u,\d)}{\longleftarrow} T^{i+1}$ are 
prolongations satisfying
$$
u^* \circ \d = \d \circ u^*
$$
where $u^*$ is the pull-back morphism of sheaves induced by $u$
for each $i$. We will denote a prolongation sequence as 
$T^*$ or $\{T^n\}_{n\geq 0}$.
Prolongation sequences naturally form a category $\mcal{C}$. 
Define $S^*$ the prolongation sequence defined by $S^i = \Spec R$
for all $i$, $u = \mathbbm{1}$ and $\d$ is the fixed 
$\pi$-derivation on $R$. Then let $\mcal{C}_{S^*}$ denote the
category of prolongation sequences defined over $S^*$.

\cblu{Similarly, a prolongation sequence $C_* = \{C_n\}_{n\geq 0}$ of 
$R$-algebras is a sequence
$$
R \stk{(u_{-1},\d_{-1})}{\longrightarrow} C_0 \stk{(u_0,\d_0)}{\longrightarrow}
C_1 \stk{(u_1,\d_1)}{\longrightarrow} \cdots ,
$$ 
where $C_n \stk{(u_n,\d_n)}{\longrightarrow} C_{n+1}$ is a $\pi$-derivation of 
$R$-algebras satisfying \eqref{proldia} and we have 
$$
u_{n+1} \circ \d_n = \d_{n+1} \circ u_n,
$$
for all $n$. Let $R_* = R \stk{(\mathbbm{1},\d)}{\longrightarrow} R
\stk{(\mathbbm{1},\d)}{\longrightarrow} \cdots$ be the prolongation sequence of
algebras with the $\pi$-derivations at all levels to the fixed $\pi$-derivation
$\d$ on $R$.

Let $D_*$ be a prolongation sequence of $R$-algebras with $D_n 
\stk{(v_n,\Delta_n)}{\longrightarrow} D_{n+1}$ to be the $\pi$-derivation at a
given level $n$.  A {\it morphism of prolongation sequences of $R$-algebras}
$\mg_* : C_* \map D_*$ is a system 
of $R$-algebra homomorphisms $g_n : C_n \map D_n$ for all $n$ given by
\begin{align}
\label{mgd0}
\xymatrix{
& \\
C_n \ar[r]^{\mg_n} \ar@{.>}[u]& D_n \ar@{.>}[u]\\
C_{n-1} \ar[r]^{\mg_{n-1}} \ar[u]^{(u,\d)} & D_{n-1} \ar[u]_{(v,\Delta)} \\
C_1 \ar[r]^{\mg_1} \ar@{.>}[u] & D_{1} \ar@{.>}[u] \\
C_0 \ar[r]^{\mg_0} \ar[u]^{(u,\d)} & D_{0}, \ar[u]_{(v,\Delta)} \\
}
\end{align}
satisfying
\begin{align*}
(i)~ g_n \circ u &= v \circ g_{n-1} \\
(ii)~ g_n \circ \d &= \Delta \circ g_{n-1}
\end{align*}
for all $n$. Let $\mcal{C}_{R_*}$ denote the category of prolongation sequences
of $R$-algebras that admit a map from $R_*$.
}

For any $S$-scheme $X$ and for all $n \geq 0$ we define the 
$n$-th jet space functor $J^nX$ as  
$$
J^nX(B):= X(W_n(B)) = \Hom_S(\Spec (W_n(B)), X)
$$
for any $R$-algebra $B$. Then $J^nX$ is representable by a 
scheme over $S$ (This was shown in \cite{bor11b} for 
$S= \Spec \Z$ and in \cite{bps} for a general 
prolongation sequence $S^*$).
Then $J^*X:=\{J^nX\}_{n\geq 0}$ forms a prolongation sequence of $S$-schemes.
and is called the {\it canonical prolongation sequence} as 
in \cite{bui00} where $\phi: J^{n+1}X \map J^nX$ denote the lift of Frobenius
morphism for each $n$. 
\cblu{
In particular, if $X = \Spec A$ is an affine scheme where $A$ is an $R$-algebra,
then $J^nX = \Spec J_nA$ where $J_nA$ is the canonical $R$-algebra representing
the above $n$-th jet space functor $J^nX$. Then $J_*A = \{J_nA\}_{n=0}^\infty$
is called the {canonical prolongation sequence of $R$-algebras}.
}

By Proposition 1.1 in \cite{bui00}, 
$J^*X$ satisfies the universal property that for any $T^* \in
\mcal{C}_{S^*}$ and $X$ a scheme over $S$ we have 
\begin{align}
\Hom_S(T^0,X) =\Hom_{\mathcal{C}_{S^*}}(T^*,J^*X).
\end{align}
\cblu{
The above universal property in the case of prolongation sequences of 
$R$-algebras means that for any $C_* \in \mcal{C}_{R_*}$ we have 
\begin{align}
\label{univ}
\Hom_R(A,C_0) = \Hom_{\mcal{C}_{R_*}}(J_*A,C_*).
\end{align}
}

Given an $R$-algebra $B$, for any $n$, let $B^{\phi^n}$ denote
the $R$-algebra obtained by considering the structure map 
$R \stk{\phi^n}{\map} R \map B$. Then given an $S$-scheme $X$
we define $X^{\phi^n}$ as $X^{\phi^n}(B):= X(B^{\phi^n})$ for
any $R$-algebra $B$. Then it is easy to see that the above functor is
represented by the base change of $X$ over the map $\phi^n:S \map S$ given by
$X^{\phi^n} = X\times_{S,\phi^n} S$.

\section{Shifted Witt Vectors}\label{swv}
In this section we construct the general
$m$-shifted $\pi$-typical Witt vectors and describe its properties.
Note that the `$(-1)$-shifted' Witt vectors are the usual $\pi$-typical Witt 
vectors. 
\cblu{
For all $n$, let $W_n(B)$ be the $\pi$-typical Witt vectors of length $n+1$
(for details the reader may see \cite{BS_b}\cite{drin76}\cite{Lars}). 
For any $m,n \geq 0$, the 
$R$-algebra restriction map $T: W_{m+n}(B) \map W_m(B)$ is given by 
$$
T(b_0,\dots , b_{m+n}) = (b_0,\dots, b_m).
$$
Given an $R$-algebra $B$ with structure map $f:R \map B$, 
for a fixed $m$ we define the 
{\it $m$-shifted $\pi$-typical Witt vectors} to be the $R$-algebra  
$$
W_{\mn}(B) := W_{m}(R) \times_{W_m(B)} W_{m+n}(B)
$$
for all $n \geq 0$. Note that we have a bijection
\begin{align}
R^{m+1} \times B^n &\map W_m(R) \times_{W_m(B)} W_{m+n}(B) = W_{\mn}(B)
\end{align}
given by
\begin{align}
\big((r_0,\dots , r_m), (b_1, \dots , b_n)\big) &\mapsto \big((r_0,\dots , 
r_m),(f(r_0), \dots, f(r_m), b_1, \dots , b_n)\big). \nonumber
\end{align}}

Define 
$$
\Pi_{\mn}B= \left(R\times R^\phi \times
\cdots R^{\phi^m} \right) \times \left(B^{\phi^{m+1}}\times \cdots \times 
B^{\phi^{m+n}}\right)
$$ 
to be the product of rings and is also naturally an $R$-algebra.
For $n\geq 0$, set $\Pi_n B := B \times B^\phi \times \cdots \times 
B^{\phi^n}$.
Then we have a natural map $\Pi_{\mn}B \map \Pi_{m+n}B$. 
Consider the shifted ghost map 
$$
w:W_{\mn}(B) \map \Pi_{\mn}(B)
$$
given by 
$$(x_0,\dots, x_m, \dots, x_{m+n}) \map \langle x_0, x_0^q+\pi x_1, \dots, 
x_0^{q^{m+n}} + \pi
x_1^{q^{m+n-1}} + \cdots + \pi^{m+n} x_{m+n} \rangle
$$

\cblu{By definition,} $W_{\mn}(B)$ is naturally endowed with the Witt ring 
structure of addition and multiplication that makes $w$ a ring homomorphism.
%
%
We now define the following ring homomorphism on the ghost side:

\begin{enumerate}
\item
The {\it restriction} map $T_w: \Pi_{\mn}(B) \map \Pi_{[m]n-1}(B)$ as 
$$
T_w\langle z_0,\dots , z_m,z_{m+1},\dots , z_{m+n}\rangle = \langle z_0,\dots, 
z_m, z_{m+1} , \dots, z_{m+n-1} \rangle.
$$
Clearly $T_w$ is a map of $R$-algebra.

\item
The {\it Frobenius} map $F_w: \Pi_{\mn}(B) \map \Pi_{[m]n-1}(B^\phi)$  as 
$$
F_w\langle z_0 \dots z_{m-1},z_m, \dots , z_{m+n}\rangle = \langle \phi(z_0),
\dots , \phi(z_{m}), z_{m+2},\dots , z_{m+n}\rangle.
$$
Here, the $z_{m+1}$-th component gets dropped in the definition of $F_w$.
\end{enumerate}

Note that following the similar arguments as in the case of usual Witt vectors 
the $R$-algebra map $T:W_{\mn}(B) \map W_{[m]n-1}(B)$ given by 
$$
T(x_0,\dots , x_m, x_{m+1},\dots , x_{m+n}) = (x_0,\dots, x_m, x_{m+1},\dots
x_{m+n-1}).
$$

\begin{theorem}
\label{Wittone}
There exists a \cblue unique functorial \color{black}
ring homomorphism $\tF: W_{\mn}(B) \map W_{[m]n-1}(B)$ such that 
$$
\xymatrix{
W_{\mn}(B) \ar[d]_-{\tF} \ar[r]^w & \Pi_{\mn}(B)  \ar[d]^{F_w} \\
W_{[m]n-1}(B^\phi) \ar[r]^w & \Pi_{[m]n-1}(B^\phi) 
}
$$
commutes. Moreover if $\tF(x_0,\dots ,x_{m+n}) = (\tF_0,\dots, \tF_{m+n-1})$, 
then for all $0 \leq h \leq m+n-1$ we have
$$
\tF_h \equiv x_h^q \bmod \pi.
$$
\end{theorem}
\begin{proof}
We will prove the result using induction. It is also sufficient to assume 
that $B$ is $\pi$-torsion free. For $h=0$, we have $\tF_0 = \phi(x_0)$ 
which clearly satisfies the required condition. Assume that the result is 
true for $h-1$. Then we have $\tF_i = x_i^q + \pi y_i$ where $y_i \in B$ 
for all $i=0, \dots, h-1$. 
\cblu{
We have two distinct cases to consider. Suppose $h \leq m$. We have seen that
$\tF_0 = \phi(x_0)$ and hence let us assume by induction that $\tF_i = 
\phi(x_i)$ for all $i = 0, \dots, h-1$. Hence comparing the ghost coordinates
we get
\beqar
\tF_0^{q^h} + \pi \tF^{q^{h-1}} + \cdots + \pi^h \tF_h &=& 
\phi(x_0^{q^{h}} + \pi x_1^{q^{h-1}} + \cdots + \pi^h x_h) \\
& = &\phi(x_0)^{q^{h}} + \pi\phi(x_1)^{q^{h-1}} + \cdots + \pi^h \phi(x_h). 
\eeqar
Hence by the induction hypothesis along with the fact that $R$ is $\pi$-torsion
free, we have $\tF_h = \phi(x_h)$ which satisfies the required condition and 
therefore proves the required result.

Let us consider the other case when $h \geq m+1$.}
Then comparing the ghost coordinates we get
$$
\tF_0^{q^h} + \pi \tF^{q^{h-1}} + \cdots + \pi^h \tF_h = x_0^{q^{h+1}} + \pi
x_1^{q^h} + \cdots + \pi^hx_h + \pi^{h+1}x_{h+1}.
$$
Grouping terms we obtain
$$
\tF_h = \sum_{i=0}^{h-1} \pi^{i-h}\left( x_i^{q^{(h+1)-i}} - \tF_i^{q^{h-i}}
\right) + x_h^q + \pi x_{h+1}
$$
and we would be done if we can show integrality of the expression on the right 
hand side. Let $L_i = \pi^{i-h}\left(x_i^{q^{(h+1)-i}} - \tF_i^{q^{h-i}}
\right)$. By the induction hypothesis for each $i = 0,\dots , h-1$ we have 
$$
L_i = \pi^{i-h} \sum_{j=1}^{q^{h-i}} \left(\begin{array}{l} q^{h-i} \\j 
\end{array}\right) \pi^j (x_i^q)^{q^{h-i}-j}y_i^j.
$$
Then note that for each $i$ the $\pi$-valuation is
\beqar
\vpi(L_i) & \geq & i-h + \vpi(q^{h-i}) + (j -\vpi(j)) \\
&\geq & (h-i)\vpi(q) + 1 - (h-i) \mb{, since } j-\vpi(j) \geq 1 \\
& = & (h-i)(\vpi(q)-1) + 1 \\
& \geq & 1
\eeqar
and we are done.
\end{proof}

Consider the natural map $I:W_{\mn}(B) \map W_{m+n}(B)$ given by 
$$I(x_0,\dots, x_m,x_{m+1},\dots , x_{m+n}) =
(x_0,\dots, x_m,x_{m+1},\dots , x_{m+n})$$
where $x_0,\dots ,x_m \in R$.
Note that $I$ is an injection if $B$ is flat over $R$.
Also denote $I_w: \Pi_{\mn} (B) \map \Pi_{m+n}(B)$ the natural map on the 
ghost side.

\begin{proposition}
\label{comm}
Let $I: W_{\mn}(B) \map W_n(B)$ be the natural map. Then
$$
F^{m+2} \circ I = F^{m+1} \circ I \circ \tF.
$$
\end{proposition}

\begin{proof}
It is sufficient to assume $B$ is $\pi$-torsion free. Then the ghost map $w$ 
is injective and hence it is enough to check the identity on the ghost
vectors. If $I_w: \Pi_{\mn}(B) \map \Pi_{m+n}(B)$ denote the natural map (which
is inclusion since $B$ is $\pi$-torsion free), then it is sufficient to show
that
$$
F^{m+2}_w \circ I_w = F_w^{m+1} \circ I_w \circ \tF_w.
$$
Now we have 
\beqar
(F_w^{m+1} \circ I_w \circ \tF_w)\langle z_0,\dots, z_{m+n}\rangle &=& 
(F_w^{m+1} \circ I_w)\langle \phi(z_0),\dots, \phi(z_m),z_{m+2},\dots , z_{m+n}
\rangle \\
&=& F_w^{m+1}\langle \phi(z_0),\dots, \phi(z_m), z_{m+2},\dots , z_{m+n}\rangle
\\
&=&\langle z_{m+2},\dots ,z_{m+n}\rangle 
\eeqar
On the other hand,
\beqar
F_w^{m+2}\circ I_w \langle z_0, \dots , z_{m+n}\rangle &=& F^{m+2}_w\langle
z_0,\dots , z_{m+n} \rangle \\
&=& \langle z_{m+2} ,\dots , z_{m+n} \rangle 
\eeqar
and this completes the proof. 
\end{proof}

Let the tuple $(X,P)$ denote the scheme
$X$
with a marked $R$-point $P \in X(R)$. Then by composing the universal map
\[
R \stk{\exp_\d}{\longrightarrow} W_n(R)
\]
with $P$, we naturally obtain an $R$-marked point $P_n$ on 
$J^nX$ which we will denote
by
$(J^nX,P_n)$. If $X = \Spec A$ and denote $i^*:A \map R$ be the
ring map associated to $P$. If $i^*_n:A \map W_n(R)$ 
denote the ring map of $P_n$ then we have
$i^*_n= \exp_\d \circ i^*$.
Then the composition
$w \circ i_n^*: A \map \Pi_n(B)$ is given by
\begin{align}
a \mapsto \langle i^*(a) , \phi(i^*(a)) ,\dots, \phi^n(i^*(a))
\rangle.
\end{align}

For all $n \geq 0$, consider $N^{\mn}X = J^{m+n}X \times_{J^mX} S$ which is 
the following fiber product
$$\xymatrix{
J^{m+n}X \ar[d] & N^{\mn}X \ar[l]_{i} \ar[d] \\
J^mX & S. \ar[l]_-{P_m} 
}$$
Then clearly $N^{\mn}X = \Spec N_{\mn}A$ where $N_{\mn}A = J_{m+n}A 
\otimes_{J_mA} R$. Also functorially $N^{\mn}X$ can be described as 
\begin{align}
\label{kerdef}
N^{\mn}X (B) &= \{g:A \map W_{\mn}(B) | \mb{ where if } g = (g_0,\dots ,g_m,
\dots , g_{m+n}), \nonumber \\
& \hspace{5cm} \mb{ then } (g_0,\dots, g_m) = P^*_m \}
\end{align}
Also note that the usual projection map $u:J^{m+n+1}X \map J^{m+n}X$ induces
$u:N^{[m]n+1}X \map N^{\mn}X$ for all $n \geq 1$.
Now we define the the generalized lateral Frobenius as follows: 
for any $R$-algebra, define $\fra:N^{\mn}X(B) \map N^{[m]n-1}X(B)$ as 
\begin{align}
\label{latFrob}
\fra(g) = \tF \circ g, \mb{ for all } g \in N^{\mn}X(B).
\end{align}

\begin{theorem}
\label{lateral}
For each $n$, the lateral Frobenius $\mfrak{f}:N^{\mn}X \map N^{[m]n-1}X$ 
is a lift of Frobenius and satisfies 
$$
\phi^{m+n} \circ i = \phi^{m+n-1} \circ i \circ \mfrak{f}
$$
for $n \geq 2$.
\end{theorem}
\begin{proof}
The lateral Frobenius $\fra$ is a lift of Frobenius with respect to $u$ follows
from Proposition \ref{Wittone} and the compositional identity is an immediate 
consequence of Proposition \ref{comm}.
\end{proof}

Let $G$ be a group scheme and $P:S \map G$ be the identity section. Then 
for each $n$, $N^{\mn}G$ naturally forms a group scheme. Also the projection
map $u:J^{\mn}G \map J^{[m]n-1}G$ induces the map (still denoted by $u$) $u:
N^{\mn}G \map N^{[m]n-1}G$ of group schemes.

\begin{theorem}
\label{latfrobgroup}
Let $(G,P)$ be as above. Then for each $n$, the lateral Frobenius
$\mfrak{f}:N^{\mn}G \map N^{[m]n-1}G$ is a morphism of group schemes.
\end{theorem}

\section{Centered Polynomials}\label{cp}

\cblue

For any $B$-algebra $C$ with $v:B \map C$ the structure map, given a lift of Frobenius $\fra:B \map C$ with respect to $v$, we say
$\Delta$ is an {\it associated $\pi$-derivation} to $\fra$ if for all 
$b \in B$ we have 
$$\fra(b)= v(b)^q + \pi \Delta b. $$

Let $R_* = R \stk{\d}{\rightarrow} R \stk{\d}{\rightarrow} \dots $ be the 
prolongation sequence
with the $\pi$-derivations at all levels to be the fixed $\pi$-derivation $\d$ 
on $R$.
Let $B_* = \{B_n\}_{n=0}^\infty$ be a prolongation sequence of $R$-algebras 
defined over $R_*$ where for all $n$, let
$u:B_n \map B_{n+1}$ be the $R$-algebra morphisms and $\Delta: B_n \map B_{n+1}$
is the $\pi$-derivation with respect to $u$.

In our particular case, for all $n$ let $B_n= R[\bt_0,\dots,\bt_n]$ be a
polynomial ring where $\bt_i$ for each $i$, denote the tuple of variables
$\bt_i = \{t_{1i},\dots, t_{di}\}$ for some $d$. For each $n$, let 
$v: B_n \map B_{n+1}$
be the natural inclusion map and suppose the system of $R$-modules 
$B_* = \{B_n\}_{n=0}^\infty$ is a prolongation sequence of $R$-algebras where
for every $n$, the $\pi$-derivation with respect to $v$ is denoted by 
$\Delta$. By $\Delta \bt_i$ we will understand
$\Delta t_{hi}$ for any chosen $h=1,\dots, d$.
Let $\fra$ denote the lift of Frobenius associated to the 
$\pi$-derivations $\Delta$. Hence for all $n$ we have
$$
\fra (b) = v(b)^q + \pi \Delta b
$$
for any $b \in B_n$.
For the $R$-algebra $B_0$, consider the canonical prolongation sequence 
$J_*(B_0) =\{ J_n(B_0)\}_{n=0}^\infty$. The $n$-th jet $R$-algebra $J_n(B_0)$
is given by 
\begin{align}
\label{JnB}
J_n(B_0) = R[\bt_0, \bt_0' ,\dots , \bt_0^{(n)}],
\end{align}
where for each $i =0, \dots, n$, $\bt_0^{(i)} = \{t_{10}^{(i)},\dots , 
t_{d0}^{(i)}\}$ is a $d$-tuple of indeterminates. For each $n$, the canonical
$R$-algebra map $u:J_n(B_0) \map J_{n+1}(B_0)$ is an inclusion induced by
$u(\bt_0^{(i)}) = \bt_0^{(i)}$ for all $i = 0, \dots, n$. For each $n$, 
the canonical $\pi$-derivation $\nd: J_n(B_0) \map J_{n+1}(B_0)$ is given by 
$\nd (\bt_0^{(i)}) = \bt_0^{(i+1)}$ for all $i = 0, \dots ,n.$
The associated lift of Frobenius $\Psi$ is given by
\begin{align}
\label{PsiFrob}
\Psi(a) = a^q +\pi \partial a,
\end{align}
for all $a \in J_n(B_0)$ and $n$.
Here we will like to remark the exception of making $\nd$ denote the canonical
$\pi$-derivation as opposed to $\d$ and $\Psi$ as the canonical lift of 
Frobenius instead of $\phi$. The reason for this is to avoid 
notational conflict in the subsequent subsection where the results of 
this one are applied to.

The universal property of the canonical prolongation sequence in (\ref{univ}) 
implies that we have
$$
\Hom_R(B_0,B_0) \simeq \Hom_{R_*}(J_*(B_0),B_*). 
$$
Hence the identity map in $\Hom_R(B_0,B_0)$ induces the following map of 
prolongation sequences $\mg_* : J_*(B_0) \map B_*$ which is a system of 
$R$-algebra homomorphisms given by
\begin{align}
\label{mgd}
\xymatrix{
& \\
J_n(B_0) \ar[r]^{\mg_n} \ar@{.>}[u]& B_n \ar@{.>}[u]\\
J_{n-1}(B_0) \ar[r]^{\mg_{n-1}} \ar[u]^{(u,\partial)} & B_{n-1} 
\ar[u]_{(v,\Delta)} \\
J_{1}(B_0) \ar[r]^{\mg_1} \ar@{.>}[u] & B_{1} \ar@{.>}[u] \\
B_0 \ar[r]^{\mg_0 = \mathbbm{1}} \ar[u]^{(u,\partial)} & B_{0}, 
\ar[u]_{(v,\Delta)} \\
}
\end{align}
satisfying
\begin{align*}
(i)~ g_n \circ u &= v \circ g_{n-1} \\
(ii)~ g_n \circ \partial &= \Delta \circ g_{n-1}
\end{align*}
for all $n$.
For the sake of brevity, if there is no possibility of confusion, for all
$n$, we will denote $\mg_n$ as $\mg$, there by, suppressing the subscripts. 

The above map in (\ref{mgd}) of prolongation sequences $\mg_*: J_*(B_0) \map 
B_*$ is induced by 
\begin{align}
\label{univg}
\mg(\bt_0^{(i)}) = \Delta^i \bt_0,
\end{align}
for all $i$ and satisfies 
$$
g \circ \partial = \Delta \circ g.
$$

\color{black}


\begin{lemma}
\label{coord1}
Let $B_*$ be as above. If the coordinate functions
$\bt_0,\dots , \bt_n$ satisfy
$$
\fra^n(\bt_0) =  \bt_0^{q^n}+ \pi \bt_1^{q^{n-1}} + \cdots + 
\pi^n \bt_n,
$$
for each $n$, then 
$$
\bt_n= \Delta \bt_{n-1} + \sum_{i=0}^{n-2}\left( \sum_{j=1}^{q^{n-1-i}} 
\pi^{i+j-n}\left(\begin{array}{l} q^{n-1-i} \\ j \end{array}\right) 
\bt_i^{q(q^{n-1-i}-j)}(\Delta \bt_i)^j \right)
$$
\end{lemma}
\begin{proof}
Follows from a similar computation as in Proposition 2.10 in \cite{bps}.
\end{proof}

We will define $G \in R[T_1,\dots, T_k]$ to be {\it centered} if $G(0,\dots,
0) =0$.
Given $k$-elements $a_1,\dots, a_k \in B_n$ we define the
subset:
\beqar
[a_1,\dots,a_k] = \{a \in B_n~|~ a = F(a_1,\dots,a_k) \mb{ for 
some } F \in R[T_1,\dots,T_k] 
\mb { centered}\}.
\eeqar

\begin{lemma}
\label{sideal}
If $a, b \in [a_1,\dots, a_k]$ then
\begin{enumerate}
\item $a+b,~ ab \in [a_1,\dots, a_k]$

\item $\Delta a \in [a_1, \dots, a_k , \Delta a_1,\dots, \Delta a_k]$.

\item $[a_1,\dots, a_k] \subset (a_1,\dots, a_k)$.

\item If $i^*: B_n \map R[\by,\dots, \by^{(m)}]$ be a
map of $R$-algebras, if $a \in [a_1,\dots, a_k]$ then 
$$
i^* a =[i^*a_1,\dots, i^*a_k].
$$

\item If $a-b \in [a_1,\dots , a_k]$ then $\Delta a - 
\Delta b \in [b,a_1,\dots, a_k, \Delta a_1, \dots \Delta a_k]$
\end{enumerate}
\end{lemma}
\begin{proof}
$(1)$ Let $a = F(a_1,\dots, a_k)$ and $b=G(a_1,\dots, a_k)$ where $F$ and $G$
are centered. Then clearly $a+b= (F+G)(a_1,\dots, a_k)$ and $ab=FG (a_1,\dots,
a_k)$ where $F+G$ and $FG$ are both centered and we are done.

$(2)$ Let $a= F_1(a_1,\dots, a_k) + \cdots + F_n(a_1,\dots, a_k)$ where 
$F_i$s are centered monomials. Then 
$$\Delta a = \Delta F_1 + \cdots + \Delta F_n + C_\pi(F_1,\dots, F_n).$$
where $C_\pi$ is centered.

\underline{Claim}: If $F= ca_1\dots a_l$, $c \in R$ and $a_i$s need not be 
distinct, then 
$$
\Delta F  = L(a_1,\dots, a_l, \Delta a_1, \dots, \Delta a_l)
$$
for some $L$ centered.

{\underline{\it Proof of claim}:} We will prove this by induction on $l$. 
For $l=1$ it is clear.
Now assume true for $l-1$ and let $\Delta (ca_1\dots a_l) = 
L(a_1\dots a_{l-1},
\Delta a_1, \dots, \Delta a_{l-1})$ with $L$ centered. Then
\beqar
\Delta F &=& \Delta (ca_1,\dots, a_l) \\
&=& \Delta((ca_1,\dots, a_{l-1})a_l) \\
&=& (ca_1,\dots, a_{k-1})^q \Delta a_l + a_l^q\Delta
 (c a_1,\dots, a_{l-1}) + \pi\Delta a_l \Delta (ca_1\dots a_{l-1}) \\
&=& (ca_1\dots a_{l-1})^q \Delta a_l + a_l^q L(a_1,\dots, a_{l-1},\Delta
 a_1,\dots, \Delta a_{l-1})\\
& & \hspace{1cm} + \pi \Delta a_l L(a_1,\dots , a_{l-1}, \Delta a_1,\dots , \Delta a_{l-1}) 
\\
& \in& [a_1, \dots, a_l,\Delta a_1,\dots, \Delta a_l]
\eeqar
and this proves the claim.
Hence to complete the proof of (2), we have $\Delta 
a \in [a_1,\dots, a_k,\Delta a_1, \dots, \Delta a_k]$ 
since $C_\pi$ is a centered polynomial.

$(3)$ Clear from the definition.

$(4)$ Let $a \in [a_1,\dots , a_k]$. Then $a = H(a_1,\dots , a_k)$ where 
$H$ is centered. Then $i^*a = H(i^*a_1,\dots , i^*a_k)$ and we are done.

$(5)$ Since $a-b \in [a_1,\dots a_k]$ implies that $a \in [b,a_1,\dots, a_k]$.
Then we have
\beqar
\Delta (a-b)  &\in & [a_1,\dots, a_k,\Delta a_1,\dots , \Delta a_k] \\
\Delta a - \Delta b + C_\pi (a, -b) &\in & [a_1,\dots, a_k,\Delta 
a_1,\dots , \Delta a_k] \\
\eeqar
and the result follows since $a \in [b,a_1,\dots, a_k]$.
\end{proof}

\begin{lemma}
\label{phi-del}
For all $m \geq 1$ we have 
$$
\fra^{m}(a) = \pi^{m}\Delta^{m}(a) + P_{m-1}(a,\Delta a, \dots, 
\Delta^{m-1}a).
$$
where $P_{m-1}$ is centered.
\end{lemma}
\begin{proof}

For $m=1$ we have $\fra(a) = a^q +\pi \Delta a$ for all $a$ and hence we 
are done. We will prove  using induction. Let the result be true
for $m-1$. Therefore we have
$$
\fra^{m-1}(a) = \pi^{m-1}\Delta^{m-1}a + P_{m-2}(a,\dots, \Delta^{m-2}a)
$$
where $P_{m-2}$ is centered. Then we have
\beqar
\fra^m(a) &=& \fra(\fra^{m-1}(a)) \\
&=& \fra(\pi^{m-1}\Delta^{m-1}a + P_{m-2}(a,\dots, \Delta^{m-2} a))\\
&=& \fra(\pi^{m-1}\Delta^{m-1}a) + \fra(P_{m-2}(a,\dots, \Delta^{m-2}a))\\
&=& (\pi^{m-1}\Delta^{m-1}a)^q + \pi\Delta(\pi^{m-1}\Delta^{m-1}a)
+ (P_{m-2}^q + \pi \Delta P_{m-2})
\eeqar
Then $\pi\Delta P_{m-2} \in [a,\Delta a, \dots \Delta^{m-1} a]$ by 
Lemma \ref{sideal} (2). Hence 
$$(\pi^{m-1}\Delta^{m-1}a)^q + P^q_{m-2} + \pi\Delta P_{m-2} \in
[a,\Delta a, \dots , \Delta^{m-1}a].
$$
Now for the other term
\beqar
\Delta(\pi^{m-1}\Delta^{m-1}a) &=& \phi(\pi^{m-1}) \Delta^m a + 
(\Delta^{m-1}a)^q \Delta(\pi^{m-1})\\
&=& \pi^{m-1}\Delta^m a+ \pi^{m-2}(1-\pi^{(m-1)(q-1)})(\Delta^{m-1}a)^q
\eeqar
Hence substituting the above in the equation for $\fra^m(a)$ we get
$$
\fra^m(a) = \pi^m\Delta^m a+ P_{m-1}
$$
where 
$$
P_{m-1} = (\pi^{m-1}\Delta^{m-1}a)^q + P^q_{m-2} + 
\pi\Delta P_{m-2} + \pi^{m-1}(1-\pi^{(m-1)(q-1)})(\Delta^{m-1}a)^q
$$
which satisfies $P_{m-1} \in [a,\Delta a, \dots ,\Delta^{m-1}a]$ and 
we are done.
\end{proof}

\begin{proposition}
\label{coord2}
\label{coord3}
Let $B_*$ be as above. If $\bt_0,\dots , \bt_m$ satisfies
$$
 \fra^m(\bt_0) = \bt_0^{q^m}+ \pi \bt_1^{q^{m-1}} + \cdots + 
\pi^m \bt_m
$$
for all $0\leq m \leq n$,
then we have
\begin{enumerate}
\item
$\bt_m -\Delta^m\bt_0 \in [\bt_0,\dots, \Delta^{m-1}\bt_0]=
[\bt_0,\bt_1, \dots, \bt_{m-1}]. $
\item $B_m \simeq R[\bt_0,\dots , \Delta^m \bt_0]$.
\end{enumerate}

\end{proposition}
\begin{proof}
We will prove this by induction on $m$. For $m=0$ then the result is clear. Now 
assume the result is true for $m-1$. Then for all $0\leq i \leq m-1$ we have
$$
\bt_i-\Delta^i \bt_0 \in [\bt_0,\dots, \Delta^{i-1}\bt_0] =[\bt_0,\bt_1,
\dots, \bt_{i-1}].
$$
Now by induction hypothesis we have 
$$
\bt_{m-1}-\Delta^{m-1}\bt_0 \in [\bt_0,\dots, \Delta^{m-2}\bt_0] = [\bt_0,
\dots, \bt_{m-2}]
$$
Hence by Lemma \ref{sideal} (5) we have
\begin{align}
\label{cat}
\Delta \bt_{m-1} - \Delta^m \bt_0 \in [\bt_0,\dots , \Delta^{m-2}\bt_0, 
\Delta^{m-1} \bt_0] = [\bt_0,
\bt_1, \dots , \bt_{m-1}]
\end{align}
Then by Lemma \ref{coord1} along with the induction hypothesis we have 
\begin{align}
\label{cat2}
\bt_m - \Delta \bt_{m-1} \in [\bt_0,\dots, \Delta^{m-1} \bt_0] = 
[\bt_0,\dots , \bt_{m-1}]
\end{align}
Hence combining (\ref{cat}) and (\ref{cat2}) we get
$$
\bt_m - \Delta^m \bt_0 \in [\bt_0,\dots, \Delta^{m-1}\bt_0] = 
[\bt_0,\dots , \bt_{m-1}]
$$
and we are done.

$(2)$ From $(1)$ we get $B_m \simeq R[\bt_0,\dots ,\Delta^m \bt_0]$ and 
this completes the proof.
\end{proof}

\cblue

\begin{theorem}
\label{univthm}
Assume $B_*$ satisfy the conditions as in Lemma \ref{coord1}.
Then the morphism $\mg_*:J_*(B_0) \map B_*$ given in (\ref{mgd}) is an 
isomorphism in the category
of prolongation sequences over $R_*$.
\end{theorem}
\begin{proof}
It is sufficient to show that in diagram (\ref{mgd}), $g: J_n(B_0) \map B_n$
is an isomorphism of $R$-algebras for all $n$. By equation (\ref{JnB}) we
have $J_n(B_0) = R[\bt_0,\dots, \bt^{(n)}_0]$. On the other hand, by 
Proposition \ref{coord3} we get $B_m \simeq R[\bt_0,\dots , \Delta^m\bt_0]$.
Then the result follows from the fact that the $R$-algebra map $g$ in 
equation (\ref{univg}) is 
given by $g(\bt_0^{(i)}) = \Delta^i \bt_0$ for all $i=0,\dots ,n$.
\end{proof}

\color{black}

\section{Kernel of Jet spaces}\label{kj}

As before we will consider the category of schemes over $S= \Spec R$
where $R$ is a ring with a fixed $\pi$-derivation $\d$ on it.
Let $ X= \Spec B$ such that $B$ is an $R$-algebra with presentation $B = R[\bx]/
I$ where $\bx$ is a collection of indeterminates and $I$ is an ideal inside 
$R[\bx]$.
Consider the representing scheme $J^nX = \Spec J_nB$ where 
$$J_nB =
 \frac{R[\bx, \bx', \cdots , \bx^{(n)}]}{(I, \d I,  
\cdots, \d^n I )}\simeq \frac{R[\bx, \bx_1, \cdots, \bx_n]}{(I, 
P_1(I), \cdots, P_n(I))}. $$ 
and the above isomorphism is as in \cite{bps}. Define 
\begin{align}
\label{Nn}
\Nn X := J^{m+n}X \times_{J^m X} P_m
\end{align}
Then 
\[
\Nn X = \Spec~ N_{\mn}B
\]
where $N_{\mn}B= J_{m+n}B \otimes_{J_mB} R$ obtained from the following
base-change 
$$\xymatrix{
J_{m+n}B \ar[r]^{i^*_m} & N_{\mn}B \\
J_mB \ar[r]^{i^*_m} \ar[u] & R. \ar[u] 
}$$


\subsection{The Affine $n$-plane case}

Let $X= \Spec A$ where $A = R[\bx]$ and $\bx$ is as before. We will do
explicit
computation based on the chosen coordinate function $\bx$. Now for any $n$ we
have $J^nX = \Spec J_nA$ where $J_nA = R[\bx,\bx_1,\bx_2,\dots , \bx_n]
\simeq R[\bx,\bx',\dots, \bx^{(n)}]$ where on $J^nX$ we have 
$(\bx,\bx_1,\dots )$ to be the induced Witt coordinates and $(\bx, \bx',
\dots )$ are the $\d$-coordinates. Then an $R$-point
of $X$ corresponds to an $R$-algebra map $i^*: A \map R$ and the induced point
$i^*_m:J_mA \map R$ correspond to the point $P_m: \Spec R \map J^mX$. 
Then we have
\beqar
N_{\mn}A &=& J_{m+n}A \otimes_{J_mA} R\\
&=& R[\bx,\bx',\dots, \bx^{(m+n)}] \otimes_{R[\bx,\dots, \bx^{(m)}]} R \\
&\simeq & R[\bx^{(m+1)},\dots , \bx^{(m+n)}] 
\eeqar
satisfying
$$\xymatrix{
J_{m+n}A \ar[r]^{i^*_m} & N_{\mn}A \\
J_mA \ar[u] \ar[r]_{i^*_m} & R \ar[u]
}$$
where 
$$i^*_m(\bx^{(j)}) = \left\{\begin{array}{ll} i^*_m(\bx^{(j)}) & \mb{ if } 
j \leq m \\ 
\bx^{(j)} & \mb{ if } j \geq m+1. \end{array} \right.$$
For a fixed $m$, consider the lateral Frobenius $\mfrak{f}:N_{\mn}A \map 
N_{[m]n+1} A$ for $n \geq 0$ as in Theorem \ref{lateral}.
The above theorem also implies 
\[
\mfrak{f} i^* \phi^{m+1} = i^* \phi^{m+2}.
\]
Let $\Delta$ be the unique $\pi$-derivation associated to $\mfrak{f}$, that is 
$\mfrak{f}(a) = a^q + \pi \Delta (a)$.
\cblue
Hence $N_{[m]*}A := \{N_{[m]n}A\}_{n=0}^\infty$ with the $\pi$-derivation
$\Delta :N_{[m]n}A \map N_{[m]n+1}A$ for all $n$ is a prolongation sequence
over $R_*$.
\color{black}
Now the adjunction property proved in
Theorem 1.3 of \cite{bps} implies
\begin{align}
\label{adjunction}
J^nX(B) \simeq \Hom_R(J_nA, B) \simeq \Hom_R(A,W_n(B))
\end{align}
for all $R$-algebra $B$. In this case when $J_nA=R[\bx,\bx_1,\dots, \bx_n]$,
giving an $R$-algebra homomorphism from $J_nA$ to $C$ is equivalent to 
specifying the image of the generators $\{\bx,\dots, \bx_n\}$ in $C$.
Hence with respect to the Witt coordinates, the above 
isomorphism in (\ref{adjunction}) gives the following correspondence:
$$
\left(\{\bx_0,\bx_1,\dots, \bx_n\} \mapsto \{\bo_0,\bo_1,\dots, \bo_n\}\right) 
\longleftrightarrow \left(\bx \mapsto (\bo_0,\dots, \bo_n) \right)
$$
where $\bo_i \in B$ for all $i$. Then $N^{\mn}X$ is a closed subfunctor of 
$J^{m+n}X$ satisfying
$$
\xymatrix{
J^{m+n}X(B) \ar@{=}[r] & \Hom_R(J_{m+n}A,B) \ar@{=}[r] & \Hom_R(A,W_{m+n}(B)) \\
N^{\mn}X(B) \ar@{^{(}->}[u] \ar@{=}[r] & \Hom_R(N_{\mn}A,B) \ar@{^{(}->}[u] 
\ar@{=}[r] & \nHom_R(A,W_{\mn}(B)) \ar@{^{(}->}[u]
}
$$
where $\nHom_R(A,W_{\mn}(B))$ consists of elements $g\in \Hom(A,W_{\mn}(B))$ 
such that $$g(a) = (i^*_m(a),g_{m+1}(a),\dots , g_{m+n}(a))$$ where 
$i^*_m(a) = (b_0,\dots ,b_m) \in W_m(R)$ and $g_{m+1}(a),\dots ,g_{m+n}(a)
\in B$. Then giving an element $g \in N^{\mn}X(B)$ corresponds to giving a map 
sending the generators of the algebra $N_{\mn}A$ to elements in $B$ as 
follows:
$$
\{\bx_{m+1},\dots, \bx_{m+n}\} \mapsto \{\bo_{m+1},\dots, \bo_{m+n}\}.
$$
The above map then under the identification in the diagram 
corresponds to an element in 
$\nHom(A,W_{\mn}(B))$ which is given by (still denoted by $g$):
$$
g:\bx \mapsto (i^*_m(\bx),\bo_{m+1},\dots, \bo_{m+n}).
$$
We may choose the coordinate $\bx$ such that $i^*(\bx) = 0$. This induces 
$i_m^*(\bx) = (0,\dots, 0) \in W_m(R)$.
Then the above map $g$ is given by
$$
\bx \mapsto (0,\dots ,0 , \bo_{m+1},\dots , \bo_{m+n}).
$$
Consider the composition of $g$ with the ghost map $w$ 
$$
A \stk{g}{\map} W_{\mn}(B) \stk{w}{\map} \Pi_{\mn}(B)
$$
given by 
\beqar
\bx &\stk{g}{\longmapsto} & (0,\dots ,0,\bo_{m+1},\dots , \bo_{m+n}) 
\stk{w}{\longmapsto} \\
& & \langle 0,\dots ,0,
\pi^{m+1}\bo_{m+1},\pi^{m+1}\bo_{m+1}^q+ \pi^{m+2}\bo_{m+2},\dots ,
\pi^{m+1}\bo_{m+1}^{q^n}+ \cdots + \pi^{m+n}\bo_{m+n} \rangle. 
\eeqar
Now choose $B=N_{\mn}A$. Then the identity map $\mathbbm{1} \in N^{\mn}X(N_{\mn}A)
$ corresponds to the element $(\bx \mapsto (0,\dots,0,\bx_{m+1},\dots, 
\bx_{m+n})) \in \nHom(A,W_{\mn}(N_{\mn}A))$. Note that if $\fra:N_{\mn}A \map
N_{[m]n+1}A$ is the associated lift of Frobenius (associated to the generalised
lateral Frobenius defined in (\ref{latFrob})), then we have the following 
commutative diagram of rings
$$
\xymatrix{
W_{\mn}(N_{\mn}A) \ar[d]_{\fra^i} \ar[r]^w & \Pi_{\mn}(N_{\mn}A) 
\ar[d]^{F_w^{\circ i}}\\
W_{[m]n-i}(N_{[m]n+i}A) \ar[r]^w & \Pi_{[m]n-i}(N_{[m]n+i}A)
}
$$ 
where the element $(0,\dots, 0, \bx_{m+1},\dots, \bx_{m+n}) \in W_{\mn}(N_{\mn}A)$
traces the following images
$$
\xymatrix{
(0,\dots,0,\bx_{m+1},\dots ,\bx_{m+n}) \ar@{|-{>}}[d]_{\fra^i} 
\ar@{|-{>}}[r]^w & 
\langle 0,\dots, 0, \pi^{m+1}\bx_{m+1},\dots \rangle \ar@{|-{>}}[d]^{F_w^{\circ
i}} \\
(0,\dots ,0, \mfrak{f}^i(\bx_{m+1}),\dots) \ar@{|-{>}}[r]^-w & \langle 0,
\dots, 0, 
\pi^{m+1} \bx_{m+1}^{q^{i-1}}+ \cdots + \pi^{m+i}\bx_{m+i},\dots \rangle
}
$$
Hence comparing the $(m+1)$-th ghost vector component in the above diagram
we obtain 
\begin{align}
\pi^{m+1}\fra^i(\bx_{m+1}) &= \pi^{m+1}\bx_{m+1}^{q^{i-1}}+ 
\pi^{m+2}\bx_{m+2}^{q^{i-2}}+  \cdots + \pi^{m+i} \bx_{m+i} \nonumber \\
\fra^i(\bx_{m+1}) &= \bx_{m+1}^{q^{i-1}}+ \pi \bx_{m+2}^{q^{i-2}} +
\cdots + \pi^{i-1}\bx_{m+i}. \label{identity}
\end{align}

\cblue
Consider the canonical map $\mg_* : J_*(N_{[m]1}A) \map N_{[m]*}A$ 
of prolongation sequences as in (\ref{mgd})

\begin{align}
\label{mgd1}
\xymatrix{
& \\
J_n(N_{[m]1}A) \ar[r]^{\mg_n} \ar@{.>}[u]& N_{[m]n+1}A \ar@{.>}[u]\\
J_{n-1}(N_{[m]1}A) \ar[r]^{\mg_{n-1}} \ar[u]^{(u,\nd)} & N_{[m]n}A 
\ar[u]_{(v,\Delta)} \\
J_{1}(N_{[m]1}A) \ar[r]^{\mg_1} \ar@{.>}[u] & N_{[m]2}A  \ar@{.>}[u] \\
N_{[m]1}A \ar[r]^{\mg_0 = \mathbbm{1}} \ar[u]^{(u,\nd)} & N_{[m]1}A. 
\ar[u]_{(v,\Delta)} \\
} 
\end{align}


\begin{theorem}
\label{coord4}
$(i)$ The above map in (\ref{mgd1}) $\mg_*: J_*(N_{[m]1}A) \map N_{[m]*}A$ is an 
isomorphism of prolongation sequences over $R_*$.
In particular for all $n \geq 1$, the isomorphism $\mg_*$ induces an 
isomorphism 
$$
g_{n-1}: J_{n-1}(N_{[m]1}A) \simeq N_{[m]n}A
$$
of $R$-algebras.

$(ii)$ If $X= {\bb{A}}^N = \Spec R[\bx]$, then for all $m \geq 0$ we have 
$$
N^{[m]*}X \simeq J^*(N^{[m]1}X)
$$
as an isomorphism of prolongation sequences of schemes over $S^*$.
In particular for all $n \geq 1$, the above induces
an isomorphism
$$
N^{\mn}X \simeq J^{n-1}N^{[m]1}X
$$
of schemes over $S$.

\end{theorem}
\begin{proof}
$(i)$ Consider $N_{[m]*}A$ where $N_{\mn}A \simeq R[\bx_{m+1},\dots ,
\bx_{m+n}]$ as above. Then by (\ref{identity}) the coordinate functions satisfy 
$$
\mfrak{f}^i(\bx_{m+1}) = \bx_{m+1}^{q^{i-1}} + \pi \bx_{m+2}^{q^{i-2}} + 
\cdots + \pi^{i-1} \bx_{m+i}, 
$$
for all $i=0,\dots , n$. Hence by Theorem \ref{univthm} we have the required
isomorphism 
$\mg_*: N_{[m] *}A \simeq J_{*}(N_{[m]1}A)$ and we are done.

$(ii)$ Follows from setting $A = R[\bx]$ where $X = \bb{A}^N = \Spec A$.
\end{proof}

Let $X = \A^1 = \Spec A$ where $A = R[x_0]$ with the marked $R$-point $P$ given
by $x_0 = 0$. Then $J^{m+n}X = \Spec N_{[m]n}A$ where 
$N_{[m]n}A = \Spec R[x_0, x_1, \cdots ,x_{m+n}] \simeq \bb{W}_{m+n}$.  

Note that $N^{[m]1}X \simeq \A^1 \simeq \Spec R[x_{m+1}]$ and hence
by Theorem \ref{coord4} we have 
\begin{align}
N^{[m]n}X \simeq J^{n-1}(N^{[m]1}X)
\simeq \bb{W}_{n-1} \simeq \Spec R[x_{m+1},\dots , x_{m+n}]
\end{align}

Then the ghost map $w$ in Section \ref{swv} induces the morphism 
\begin{align}
\label{latghost}
N^{[m]n}X \stk{w}{\longrightarrow} \Pi_{n-1}X
\end{align}
given by 
\begin{align*}
(x_{m+1},\dots, x_{m+n}) \mapsto \langle x_{m+1}, x_{m+1}^q + \pi x_{m+2},
 \dots ,x_{m+1}^{q^{n-1}} + \pi x_{m+2}^{q^{n-2}} + \cdots  \pi^{n-1} x_{m+n} \rangle
\end{align*}

Then by Theorem \ref{Wittone} we have 
\begin{align}
\label{fraghost}
\xymatrix{
N^{[m]n+1}X \ar[r]^-w \ar[d]_-\fra & \Pi_{n}X \ar[d]^-{\fra_w}\\
N^{[m]n}X \ar[r]^-w & \Pi_{n-1}X
}
\end{align}
where $\fra_w\langle z_{m+1},\dots, z_{m+n}\rangle= \langle z_{m+2},\dots ,
z_{m+n} \rangle$ for all $\langle z_{m+1}, \dots , z_{m+n}\rangle \in 
\Pi_n X$.

\begin{corollary}
\label{coord5}
Let $J \subset N_{[m]1}A$ be an ideal. For all $n$, $\mg_*$ induces the 
isomorphism
$$
g_{n}: (J,\nd J, \dots, \nd^{n} J) \simeq (J, \Delta J, \dots , 
\Delta^{n}J).
$$
\end{corollary}
\begin{proof}
Consider the prolongation sequence $N_{[m]*}A$. Then the lift of Frobenius 
$\fra$ satisfies the condition in Proposition \ref{coord3} by setting $\bt_i
= \bx_{i+m+1}$ for all $i$. Hence by Proposition \ref{coord3} $(2)$ we have
for all $n$ 
$$
N_{[m]n}A \simeq R[\bx^{m+1},\dots , \bx^{m+n}] \simeq R[\bx^{m+1}, 
\Delta \bx^{m+1} ,\dots , \Delta^{n-1} \bx^{m+1}].
$$
For all $n$, we have $J_{n-1}(N_{[m]1}A) = R[\bx^{m+1}, \nd(\bx^{m+1}), \dots
,\nd^{n-1}(\bx^{m+1})]$.
The isomorphism $\mg_*$ in Theorem \ref{coord4} induces the following 
isomorphism of prolongation sequences
$$\xymatrix{
& \\
J_{n}(N_{[m]1}A) = R[\bx^{m+1}, \dots , \nd^{n}(\bx^{m+1})] 
\ar[r]^-g \ar@{.>}[u] & R[\bx^{m+1},\dots \Delta^{n}(\bx^{m+1})] \simeq 
N_{[m]n+1}A \ar@{.>}[u] \\
J_{n-1}(N_{[m]1}A) = R[\bx^{m+1}, \dots , \nd^{n-1}(\bx^{m+1})] 
\ar[u]^{\nd}
\ar[r]^-g & R[\bx^{m+1},\dots \Delta^{n-1}(\bx^{m+1})] \simeq N_{[m]n}A
\ar[u]_\Delta \\
\ar@{.>}[u] & \ar@{.>}[u]\\
}$$
where $g(\nd^i(\bx^{m+1})) = \Delta^i(\bx^{m+1})$ for all $i$.
Hence for all $f \in N_{[m]1}A$ we have $g (\partial^i(f)) = \Delta^i(f)$
for all $i$ and hence proves the result.
\end{proof}

We will abbreviate $i^*_m$ as $i^*$ for the next proposition.

\begin{proposition}
\label{Delta-del}
Let $b \in J_mA$ be such that
$i^*b=i^*\d b= \cdots  = i^*\d^m b =0 $. Then for all $n \geq 1$ we have 
$$
\Delta^n i^*\d^{m+1}b - i^*\d^{m+n+1}b \in [i^*\d^{m+1}b,\dots, i^*\d^{m+n}b].$$
More so the above implies that 
$$
[i^*\d^{m+1}b,\dots, i^*\d^{m+n+1}b] = [i^*\d^{m+1}b, \Delta i^*\d^{m+1}b,\dots
,\Delta^n i^*\d^{m+1}b].
$$
\end{proposition}
\begin{proof}
For all $a \in J_nA$ by Theorem \ref{lateral} we have 
$$
i^* \phi^{m+2}(a) = \mfrak{f} i^*  \phi^{m+1}(a).
$$
By Lemma \ref{phi-del} we have 
$$\phi^{m+2}(a) = \pi^{m+2}\d^{m+2}(a) + P_{m+1}(a,\d a, \dots, \d^{m+1}a).
$$
for some centered polynomial $P_{m+1}(T_0,\cdots T_{m+1})$ defined over $R$. 
Applying $i^*$ to the above relation we obtain
\beqar
i^*\phi^{m+2}(a) &=& \pi^{m+2}i^*\d^{m+2}(a) + i^*P_{m+1}(a,\d a, \dots 
\d^{m+1}a)\\
&=& \pi^{m+2}i^*\d^{m+2}(a) + P_{m+1}(i^*a,i^*\d a, \dots i^*\d^{m+1}a)
\eeqar
Now note that applying Lemma \ref{phi-del} again we get
\beqar
\mfrak{f}i^* \phi^{m+1}(a) &=& \mfrak{f}i^*(\pi^{m+1}\d^{m+1}a + P_m(a,\d a,
\dots , \d^m a))\\
&=& \pi^{m+1}\fra(i^*\d^{m+1}a) + \fra(P_m(i^*a,\dots , i^*\d^m a))\\
&=& \pi^{m+1}((i^*\d^{m+1}a)^q + \pi\Delta i^*\d^{m+1}(a)) + P_m(i^*a,\dots 
,i^*\d^ma)^q \\
& & \hspace{4.5cm} + \pi\Delta P_m(i^*a,\dots , i^*\d^m a) \\
&=& \pi^{m+2}\Delta i^*\d^{m+1}(a) + \pi^{m+1}(i^*\d^{m+1}a)^q + 
P_m(i^*a,\dots ,i^*\d^ma)^q \\
& & \hspace{4.5cm} + \pi\Delta P_m(i^*a,\dots , i^*\d^m a) 
\eeqar
Hence combining the above two we get
\beqar
\pi^{m+2}i^*\d^{m+2}(a) &=& 
\pi^{m+2}\Delta i^*\d^{m+1}(a) + \pi^{m+1}(i^*\d^{m+1}a)^q \\
 + P_{m+1}(i^*a,i^*\d a, \dots i^*\d^{m+1}a) & &  + P_m(i^*a,\dots ,i^*\d^ma)^q  + \pi\Delta P_m(i^*a,\dots , i^*\d^m a) \\
\eeqar
This gives us
\beqar
i^*\d^{m+2}a &=& \Delta i^*\d^{m+1}a + \frac{1}{\pi^{m+2}}(P_m(i^*a,\dots 
i^*\d^m a)^q + \pi^{m+1}(i^*\d^{m+1}a)^q \\
& & \hspace{1.5cm} + \pi\Delta P_m(i^*a,\dots ,i^*\d^m a) - P_{m+1}(i^*a,\dots , i^*\d^{m+1}a)) \hspace{.5cm}\cdots~ (*)
\eeqar
Substituting $a$ for $b$ and $i^*b=i^*\d b= \cdots = i^*\d^m b =0$ in the above
we obtain
\beqar
i^*\d^{m+2}b &=& \Delta i^*\d^{m+1}b + \frac{1}{\pi^{m+2}}(\pi^{m+1}(i^*\d^{m+1}
b)^q - P_{m+1}(0,\dots, 0, i^*\d^{m+1}b))\\
&=& \Delta i^* \d^{m+1}b + H_1(i^*\d^{m+1}b)
\eeqar
where $H_1(t) = \frac{1}{\pi^{m+2}}(\pi^{m+1}t^q - P_{m+1}(0,\dots,0, t))$
and this shows the result is true for $n=1$.
Now we will use induction to prove our result  and assume that it holds 
true for $n-1$. Then we have 
$$
\Delta^{n-1}i^*\d^{m+1}b - i^*\d^{m+n}b \in [i^*\d^{m+1}b,i^*\d^{m+2}b,\dots
,i^*\d^{m+n-1}b]
$$
which in turn implies 
\begin{align}
\label{sim}
[i^*\d^{m+1}b,\dots , i^*\d^{m+n}b] = [i^*\d^{m+1}b, \Delta i^*\d^{m+1}b,
\dots, \Delta^{n-1}i^*\d^{m+1}b],
\end{align}
since $[i^*\d^{m+1}b,i^*\d^{m+2}b,\dots
,i^*\d^{m+n-1}b] = [i^*\d^{m+1}b,\Delta i^*\d^{m+1}b,\dots
,\Delta^{n-2} i^*\d^{m+1}b]$ by the induction hypothesis.
Let
\begin{align}
\label{si}
\Di{n-1}{m+1}b - \di{m+n} b = L(\di{m+1}b,\Di{}{m+1}b,\dots, \Di{n-2}{m+1}b)
\end{align}
for some centered $L$.  Then applying $\Delta$ to 
(\ref{si}) gives us
\beqar
\Di{n}{m+1}b-\Di{}{m+n}b+ & & \\
 C_\pi(\Di{n-1}{m+1}b, -\di{m+n}b) &=&
\Delta L(\di{m+1}b,\Di{}{m+1}b,\dots , \Di{n-2}{m+1}b)\\
&=& M(\di{m+1}b ,\dots, \Di{n-1}{m+1}b) \mb{ by Lemma \ref{sideal}(2)}
\eeqar
for some centered $M$ of degree $\geq 1$. Therefore we obtain
\beqar
\Di{n}{m+1}b-\Di{}{m+n}b &=& M(\di{m+1}b,\dots, \Di{n-1}{m+1}b)\\
& & -C_\pi(\Di{n-1}{m+1}b,-\di{m+n}b) \\
\eeqar
Note that substituting $\di{m+n}b$ by $\Di{n-1}{m+1}b -L(\di{m+1}b,\dots,
\Di{n-2}{m+1}b)$ as in (\ref{si})  we observe that 
$$
C_\pi(\Di{n-1}{m+1}b, -\di{m+n}b) = D(\di{m+1}b,\dots, \Di{n-1}{m+1}b)
$$
for some centered $D$ of degree $\geq 1$. Hence we obtain 
\begin{align}
\label{kal}
\Di{n}{m+1}b-\Di{}{m+n}b = H(\di{m+1}b, \dots, \Di{n-1}{m+1}b)
\end{align}
where $H= M-D$. Now substitute $a=\d^{n-1}b$ in $(*)$. Then we have
\beqar
\di{m+n+1}b &=& \Di{}{m+n}b + \frac{1}{\pi^{m+2}}(P_m(\di{n-1}b,\dots, 
\di{m+n-1} b)^q + \pi^{m+1}(i^*\d^{m+n}b)^q \\
& &+ \pi \Delta P_m(\di{n-1}b, \dots, \di{m+n-1}b) -P_{m+1}(\di{n-1}b,\dots,
i^*\d^{m+n}b))
\eeqar
Then setting
$$
i^*b = i^*\d b= \cdots = i^*\d^m b = 0
$$
in the above relation we obtain 
\begin{align}
\label{som}
\Di{}{m+n}b - \di{m+n+1}b =  \tilde{K}(\di{m+1}b,\dots, \di{m+n}b)
\end{align}
for some centered $\tilde{K}$. But since $[\di{m+1}b,\dots , \di{m+n}b]
= [\di{m+1}b,\dots, \Di{n-1}{m+1}b]$ as in (\ref{sim}) we must have 
\begin{align}
\label{can}
\tilde{K}(\di{m+1}b,\dots, \di{m+n}b) = K(\di{m+1}b,\dots, \Di{n-1}{m+1}b)
\end{align}
for some centered $K$. Therefore adding (\ref{kal}) and (\ref{som}) along with
substituting (\ref{can}) in it we obtain
\beqar
\Di{n}{m+1}b - \di{m+n+1}b &=&K(\di{m+1}b,\dots, \Di{n-1}{m+1}b) \\
& & \hspace{2cm} + H(\di{m+1}b,\dots ,\Di{n-1}{m+1}b) 
\eeqar
Hence we have 
$$
\Di{n}{m+1}b - \di{m+n+1}b \in [\di{m+1}b,\dots, \Di{n-1}{m+1}b]
$$
and therefore we also have
$$
\Di{n}{m+1}b - \di{m+n+1}b \in [\di{m+1}b,\dots, \di{m+n}b].
$$
Then the above clearly implies 
$$
[\di{m+1}b,\dots, \Di{n}{m+1}b] = [\di{m+1}b,\dots \di{m+n+1}b]
$$
and we are done.
\end{proof}

\begin{corollary}
\label{sameideal}
For all $m$ and $n$ we have
$$
(\di{m+1}b,\dots , \Di{n}{m+1}b) = (\di{m+1}b,\dots \di{m+n+1}b).
$$
\end{corollary}
\begin{proof}
This follows immediately from the fact that for all $m$ and $n$ by
Proposition \ref{Delta-del} we have
$$[\di{m+1}b,\dots, \Di{n}{m+1}b] = [\di{m+1}b,\dots , \di{m+n+1}b].$$
\end{proof}

For all $n$ define
\begin{align}
\nonumber
D_n = \frac{R[\bx^{m+1},\dots, \nd^n \bx^{m+1}]}{(J,\nd J, \dots,\nd^{n}J)}.
\end{align}
Then $\partial : R[\bx^{m+1},\dots, \nd^{n-1} \bx^{m+1}] \map 
R[\bx^{m+1},\dots , \nd^{n} \bx^{m+1}]$ naturally descends to 
$\pi$-derivations
$$
\partial : \frac{R[\bx^{(m+1)},\dots, \nd^{n-1} \bx^{(m+1)}]}{(J, \nd J,
\dots, \nd^{n-1} J)} \longrightarrow
\frac{R[\bx^{(m+1)},\dots, \nd^n\bx^{(m+1)}]}{(J, \nd J,
\dots, \nd^{n} J)}
$$
for all $n$. Hence the above $\pi$-derivations $\nd$ make 
$D_* :=\{D_n\}_{n=0}^\infty$ into a prolongation sequence.

For all $n$ define
\begin{align}
\nonumber
C_n = \frac{R[\bx^{m+1},\dots, \Delta^n \bx^{m+1}]}{(J,\Delta J, \dots,
\Delta^{n}J)}.
\end{align}
Then $\Delta : R[\bx^{m+1},\dots, \Delta^{n-1} \bx^{m+1}] \map 
R[\bx^{m+1},\dots , \Delta^{n} \bx^{m+1}]$ naturally descends to
$\pi$-derivations
$$
\Delta : \frac{R[\bx^{(m+1)},\dots, \Delta^{n-1} \bx^{(m+1)}]}{(J, \Delta J,
\dots, \Delta^{n-1} J)} \longrightarrow
\frac{R[\bx^{(m+1)},\dots, \Delta^n\bx^{(m+1)}]}{(J, \Delta J,
\dots, \Delta^{n} J)}
$$
for all $n$. Hence the above $\pi$-derivations $\Delta$ make
$C_* :=\{C_n\}_{n=0}^\infty$ into a prolongation sequence.

Hence the natural map $\mg_*$ descends naturally to a map of prolongation
sequences (still denoted by ) $\mg_* : D_* \map C_*$.

\begin{theorem}
\label{coord6}
The map $\mg_*: D_* \map C_*$ is an isomorphism of prolongation sequences. 
In particular for all $n$ we have
$$
g: D_n =  \frac{R[\bx^{(m+1)},\dots ,\nd^{n}\bx^{(m+1)}]}{(J,\nd J,\dots,
\nd^nJ)} \map 
\frac{R[\bx^{(m+1)},\dots ,\Delta^{n}\bx^{(m+1)}]}{(J,\nd J,\dots,
\Delta^nJ)} = C_n
$$
is an isomorphism of $R$-algebras.
\end{theorem}
\begin{proof}
The result follows from Theorem \ref{coord4} and Corollary \ref{coord5}.
\end{proof}

\color{black}

\subsection{The case for Affine Schemes}
Let $S = \Spec R$. Consider $I \subset R[\bx]$ to be an ideal and let 
$(I,\d I,\dots, \d^n I)$ denote 
the ideal generated by $\d^i f, f \in I, i=0, \dots , n$. Let $i^*:R[\bx] 
\map R$ be a ring homomorphism such that $I \subset \ker i^*$, that is 
$i^*$ corresponds to an element $P: \Spec R \map \Spec~ (R[\bx]/I)$.

Our next aim is to show that the above Theorem \ref{coord4} 
is true for any affine $X = \Spec 
R[\bx]/I$. Note that  any $R$-point $P \in X(R)$ corresponds to a 
surjective map $i^*:R[\bx] \map R$ such that $I \subset \ker(i^*)$. By the
universal property of jet rings the induced $R$-point $P_m \in J^mX$ 
corresponds to the $R$-algebra map 
$i^*_m:R[\bx,\dots, \bx^{(m)}]
\map R$ given by $i^*_m(\bx^j) = \d^j (i^*(\bx))$ for all $j =0, \dots, m$.
Clearly we have  $(I,\d I, \dots, \d^m I) \subset \ker(i^*_m)$.
Set $B= R[\bx]/I$. Then $J_mB = \frac{R[\bx,\bx',\dots, \bx^{(m)}]}
{(I,\d I, \dots, \d^m I)}$ and we have the induced morphism 
$$ i^*_m:J_mB \map R.$$ 
Recall the following tensor product of algebras
$$
\xymatrix{
J_{m+n}B \ar[r]^{i^*_m} & N_{\mn}B \\
J_mB \ar[r]^{i^*_m} \ar[u] & R \ar[u]
}$$
where $N_{\mn}B:= J_{m+n}B \otimes_{J_mB} R$ and $N^{\mn}X:= 
\Spec N_{\mn}B$.
Then  $N^{\mn}X = J^{m+n}X \times_{J^mX} S$.

\begin{lemma}
\label{N\mn}
For all $m$ and $n$ we have  
$$
N_{\mn}B \simeq \frac{R[\bx^{(m+1)},\dots , \bx^{(m+n)}]}{(i^*\d^{m+1}I,\dots,
i^*\d^{m+n}I)}.
$$
\end{lemma}
\begin{proof}
Directly follows from above.
\end{proof}


\cblue
\begin{corollary}
\label{samequot}
For $B$ as above and all $n \geq 1$ we have 
$$
N_{[m]n}B = 
\frac{R[\bx^{m+1},\dots, \bx^{(m+n)}]}{(i^*_m\d^{m+1}I, \dots,
i^*_m \d^{m+n}I)} \simeq
\frac{R[\bx^{(m+1)},\dots, \bx^{(m+n)}]}{(i^*_m\d^{m+1}I,
\dots, \Delta^{n-1}i^*_m\d^{m+1} I)}
$$
\end{corollary}
\begin{proof}
This immediately follows from Corollary \ref{sameideal}.
\end{proof}

Recall for all $n$, the canonical $\pi$-derivation 
$$\Delta: R[\bx^{(m+1)},\dots ,\bx^{(m+n)}] \map R[\bx^{(m+1)}, \dots , 
\bx^{(m+n+1)}]$$ 
in the case of affine $N$-space. 

By Corollary \ref{sameideal}, for all $n$, the $\pi$-derivation $\Delta$ 
descends uniquely to a $\pi$-derivation on the quotient rings
\begin{align}
\label{Delta-desc}
\Delta: N_{[m]n}B = 
\frac{R[\bx^{m+1},\dots, \bx^{(m+n)}]}{(i^*_m\d^{m+1}I, \dots,
i^*_m \d^{m+n}I)} \map 
\frac{R[\bx^{m+1},\dots, \bx^{(m+n+1)}]}{(i^*_m\d^{m+1}I, \dots,
i^*_m \d^{m+n+1}I)} = N_{[m]n+1}B.
\end{align}
and satisfies $\fra(b) = b^q + \pi \Delta (b)$.
Hence the above $\pi$-derivations $\Delta$ make $N_{[m]*}B := 
\{N_{[m]n}B\}_{n=1}^\infty $ into a prolongation sequence over $R_*$.

The universal property for canonical prolongation sequences (\ref{univ}) 
implies that we have
$$
\Hom_R(N_{[m]1}B, N_{[m]1}B) \simeq \Hom_{R_*}(J_*(N_{[m]1}B),N_{[m]*}B). 
$$ 
Hence the identity map $g_0:= \mathbbm{1} \in 
\Hom_R(N_{[m]1}B,N_{[m]1}B)$ induces the following map of
prolongation sequences $\mg_* : J_*(N_{[m[1]}B) \map N_{[m]*}B$ which is a system of 
$R$-algebra homomorphisms given by
\begin{align}
\label{mgd2}
\xymatrix{
& \\
J_n(N_{[m]1}B) \ar[r]^{\mg_n} \ar@{.>}[u]& N_{[m]n+1}B \ar@{.>}[u]\\
J_{n-1}(N_{[m]1}B) \ar[r]^{\mg_{n-1}} \ar[u]^\partial & N_{[m]n}B 
\ar[u]_\Delta \\
J_{1}(N_{[m]1}B) \ar[r]^{\mg_1} \ar@{.>}[u]^\partial & N_{[m]2}B  
\ar@{.>}[u]_\Delta \\
N_{[m]1}B \ar[r]^{\mg_0 = \mathbbm{1}} \ar[u]^\partial & N_{[m]1}B. 
\ar[u]_\Delta \\
} 
\end{align}
where $\partial: J_n(N_{[m]1}B) \map J_{n+1}(N_{[m]1}B)$ for all $n$, is
the canonical $\pi$-derivation for the canonical prolongation sequence
$J_*(N_{[m]1}B)$.

\begin{corollary}
\label{affjetalg}
For $B$ as above, for all $n \geq 1$ we have 
$$
N_{\mn}B \simeq \frac{R[\bx^{(m+1)},\dots, \Delta^{n-1}\bx^{(m+1)}]}
{(i^*_m\d^{m+1}I, \dots, \Delta^{n-1}i^*_m\d^{m+1} I)}.
$$
\end{corollary}
\begin{proof}
Since $N_{[m]1}B = R[\bx^{(m+1)}]/(i^*_m\d^{m+1}I)$ we have
\beqar
N_{[m]n}B &=& \frac{R[\bx^{(m+1)},\dots, \bx^{(m+n)}]}{(i^*_m\d^{m+1}I,
\dots, i^*_m\d^{m+n} I)} \\
&= & \frac{R[\bx^{m+1},\dots, \Delta^{n-1}\bx^{(m+1)}]}{(i^*_m\d^{m+1}I, \dots, 
\Delta^{n-1}i^*_m \d^{m+1}I)} \hspace{.5cm} \mb{(by Theorem 
\ref{coord4} and Corollary \ref{sameideal})}
\eeqar
and we are done.
\end{proof}




\begin{proof}[\cbl{\bf Proof of Theorem \ref{Niso1}}]
Consider the canonical prolongation sequence $J_*(N_{[m]1}B)=
\{J_n(N_{[m]1}B)\}_{n=0}^\infty$, which for all $n$ is given by 
$$
J_n(N_{[m]1}B) = \frac{R[\bx^{m+1},\nd \bx^{m+1}, \dots , \nd^{n-1} \bx^{m+1}]}
{(i^*_m \d^{m+1} I, \nd i^*_m \d^{m+1}I, \dots, 
\nd^{n-1} i^*_m \d^{m+1}I) },
$$
where $\nd : J_n(N_{[m]1}B) \map J_{n+1}(N_{[m]1}B)$ is the canonical 
$\pi$-derivation for the canonical prolongation sequence $J_*(N_{[m]1}B)$
given by $\nd (\nd^i\bx^{m+1})= \nd^{i+1}\bx^{m+1}$ for all $i = 0,1 ,\dots$.

By Corollary \ref{affjetalg}, we have 
$$
N_{[m]n}B = \frac{R[\bx^{m+1},\dots, \Delta^{n-1}\bx^{(m+1)}]}{(i_m^*\d^{m+1} I,
\dots , \Delta^{n-1}i^*_m \d^{m+1}I)}
$$
Hence the result follows from Theorem \ref{coord6} by setting the ideal $J = 
i^*_m\d^{m+1}I$,  $D_* = J_*(N_{[m]1}B)$ and $C_* = N_{[m]*}B$.
\end{proof}






\color{black}

\subsection{The $\pi$-formal Schemes case}
Let $\Ou$ is a Dedekind domain of characteristic $0$ and
$R$ is a $\pi$-adically complete discrete valuation ring. Let $l=R/\pi R$ be
its residue field. Denote $S = \Spf R$.
Let $X$ be a $\pi$-formal scheme over $\Spf R$ where $R$ is now $\pi$-adically
complete and $P: \Spf R \map X$ be an $R$-marked point. For all $n\geq 1$, let $u: J^nX \map X$ be the natural projection 
morphism.

\begin{lemma}
\label{uinverse}
For any open $\pi$-formal subscheme $U \inj X$ we have 
$$
u^{-1}U := U \times_X J^nX \simeq J^nU.
$$
\end{lemma}

\begin{proof}
This is Proposition 1.7 of \cite{bui00}.
\end{proof}

\cblue



\begin{proof}[\cbl{\bf Proof of Theorem \ref{formalcase1}}]
Let $U \inj X$ be an open affine $\pi$-formal scheme such that the marked point
$P:\Spf R \map X$ factors through $U$. Hence by lemma \ref{uinverse} we
have that the morphism $P^m: \Spf R \map J^mX$ factors through the open affine 
$\pi$-formal scheme $J^mU$. Then for each $n$ we have
$$
N^{\mn}X:= J^{m+n}X \times_{J^mX} P^m = J^{m+n}U \times_{J^mU} P^m.
$$
Hence it is sufficient to assume $X$ is affine and then the result follows
from Theorem \ref{affjet1}.
\end{proof}

\color{black}

Let $\bb{W}_n$ denote the Witt vectors of length $n+1$ 
as a group (in fact ring) object in
the category of $\pi$-formal schemes. Then $\bb{W}_n \simeq \Spf R[x_0,\dots,
x_n]\h$ and the ring $R[x_0,\dots,x_n]\h$ has a coalgebra structure induced
from the Witt vector addition. For each $n$ let us denote $N^nG :=N^{[0]n}G$.

\begin{lemma}
\label{JetGa}
For all $n$ we have $J^n\hG \simeq \bb{W}_n$.
\end{lemma}
\begin{proof}
Let $x$ be a coordinate at the origin of $\hG$. Then $\hG \simeq \Spf R[x]\h$.
Hence  given any $R$-algebra $B$ we have,
$$J^n\hG (B) = \hG(W_n(B)) = \Hom_R(R[x]\h, W_n(B)) \simeq \bb{W}_n(B)$$
and we are done.
\end{proof}


\begin{proof}[\cbl{\bf Proof of Theorem \ref{JetGg}}]
Since $G$ is a smooth $\pi$-formal group scheme, we have the following 
short exact sequence 
$$
0 \map N^{\mn}G \map J^{m+n}G \map J^mG \map 0
$$
of $\pi$-formal group schemes.
By Proposition 2.2 and Lemma $2.3$ in \cite{bui95} we have $N^{[m]1}G \simeq 
(\hG)^d$. Hence by Theorem \ref{formalcase1} and Lemma \ref{JetGa} 
we get 
$$N^{\mn}G \simeq J^{n-1}(N^{[m]1}G) \simeq (J^{n-1}\hG)^d \simeq (\bb{W}_{n-1})^d
$$
and we are done.
\end{proof}

\begin{corollary}
\label{JetG}
Let $R$ and $G$ be as above.
Then $N^nG \simeq \left(\bb{W}_{n-1}\right)^d$ and $J^nG$
satisfies the following short exact sequence of $\pi$-formal group schemes
$$
0 \map \left(\bb{W}_{n-1}\right)^d \map J^nG \map G \map 0.
$$
\end{corollary}
\begin{proof}
It follows directly from Theorem \ref{JetGg} for $m=0$.
\end{proof}

\section{Delta Isocrystals and their Integral Models}
\label{intmod}
From now on, $\Ou$ is a Dedekind domain of characteristic $0$ and
$R$ is a $\pi$-adically complete discrete valuation ring with a 
$\pi$-derivation $\d$ lifting the one on $\Ou$. Let $l= R/\pi R$ be its 
residue field. Also assume $\vpi(p) \leq p-2$. Denote $S = \Spf R$.
Let $G$  be a $\pi$-formal group scheme over $S$ of relative dimension $g$. 
From this section onwards, we will restrict ourselves to the case $m=0$. For 
any group scheme $G$, let us denote $J^nG := J^{[0]n}G$ and $N^nG= N^{[0]n}G$.

\cblue
Hence for all $n \geq 1$ by Theorem \ref{formalcase1}, we have $N^nG \simeq J^{n-1}(N^1G)$ and hence the lateral Frobenius $\fra: N^nG \map N^{n-1}G$ is the 
associated Frobenius map of the canonical prolongation sequence.
\color{black}

\subsection{Delta Characters of Group Schemes}
We recall some basic results on $\d$-characters from Sections $7,8$ of 
\cite{BS_b} that led to the construction of a natural isocrystal $\Hd(A)$ over 
$K$, which we call the $\d$-isocrystal associated to an abelian scheme $A$. We will now show 
that $\Hd(A)$ is a non-degenerate isocrystal in the case 
when $A$ is an elliptic curve over $\bb{Z}_p$. 

Let $T^*$ be a prolongation sequence. For any $s \geq 0$, define the {\it 
shifted prolongation sequence} to be $T^{*+s}= \{T^{s+n}\}_{n=0}^\infty$.
Then a morphism $\Theta: J^nG \map \hG$ is called a {\it $\d$-character of 
$G$ of order $\leq n$}. By the universal property of jet spaces, such a 
$\Theta$ induces a morphism of prolongation sequences $\Theta: J^{*+n}G \map
\hG$.

Define a {\it $\delta$-character of order $n$}, $\Theta:G \map \hG$ 
to be a $\d$-morphism of order $n$ from $G$ to $\hG$,
which is also a group homomorphism of $\pi$-formal group schemes.
By the universal property of jet schemes as in Proposition \ref{univ},
an order $n$ $\d$-character is equivalent to a homomorphism
$\Theta:J^nG \map \hG$ of $\pi$-formal group schemes over $S$. 
We denote the group of 
$\d$-characters of order $n$ by $\bX_n(G)$:
$$
\bX_n(G)=\Hom_S(J^nG,\hG).
$$
Note that $\bX_n(G)$ comes with an
$R$-module structure via $\hG$.
For any $R$-module $M$, let us denote
$$
M_{\phi} = R\otimes_{\phi,R}M.
$$
Then recall that the $\phi$-linear map $\phi^*:\bX_{n-1}(G) \to \bX_n(G)$ induces a linear map
$\bX_{n-1}(G)_{\phi}\to \bX_n(G)$, which we will abusively also denote $\phi^*$.
We then define
$$
\bH_{n}(G) = \frac{\Hom(N^{n}G,\hG)}{i^*\phi^*(\bX_{n-1}(G)_{\phi})}.
$$
Note that 
$u:N^{n+1}G\map N^nG$ induces $u^*:\Hom(N^nG,\hG) \map \Hom(N^{n+1}G,\hG)$. 
Moreover, since $u$ commutes with both $i$ and $\phi$, we have
$$
u^*i^*\phi^*(\bX_n(G)) = i^*\phi^*u^*(\bX_n(G)) \subset i^*\phi^*(\bX_{n+1}(G)),
$$
and hence $u$ also induces a map $u^*:\bH_n(G) \map \bH_{n+1}(G)$.
Define 
\begin{equation}
	\label{bigH}
	\Hd(G)= \varinjlim \bH_n(G)
\end{equation}
where the limit is taken in the category
of $R$-modules. Similarly, $\mfrak{f}: N^{n+1}G \map N^nG$ induces 
$\mfrak{f}^*:\Hom(N^nG,\hG) \map \Hom(N^{n+1}G,\hG)$, 
which descends to a $\phi$-semilinear morphism of $R$-modules
\begin{equation}	
	\label{latfrobH}
	\mfrak{f}^*:\bH_n(G) \map\bH_{n+1}(G)
\end{equation}
because we have $\mfrak{f}^*i^*\phi^*(\bX_{n-1}(G))= 
i^*\phi^*\phi^*(\bX_{n-1}(G)) \subset i^*\phi^*\bX_{n}(G)$.
This in turn induces a $\phi$-semilinear endomorphism $\mfrak{f}^*:\Hd(G) \map \Hd(G)$. 

\cblue
Let $\mathrm{pr}_j: \hG^g \map \hG$ denote the $j$-th projection for all
$j = 1 ,\dots , g$. In Lemma $2.3$ of \cite{bui95}, Buium constructs an 
isomorphism of $\pi$-formal group schemes 
\begin{align}
\label{Psi_1}
\Psi_1 : N^1G \map \hG^g
\end{align}
that depends on a choice of \'{e}tale coordinates on $G$. Then $\Psi_1$ 
can be represented as a $g$-tuple of maps
$$
\Psi_1 = (\Psi_{11}, \dots , \Psi_{1g})
$$
where $\Psi_{1j}:= \mathrm{pr}_j \circ \Psi_1$ for all $j= 1, \dots , g$.

For all $i=1,\dots , n$ define $\Psi_i$ as the composition 
\begin{align}
\label{Psi_i}
N^nG \stk{\fra^{\circ (i-1)}}{\longrightarrow} N^{n-i+1}G 
\stk{u^{\circ (n-i)}}{\longrightarrow} N^1G \stk{\Psi_1}{\longrightarrow}
\hG^g
\end{align}
where $u:N^jG \map N^{j-1}G$ is the usual projection map induced from the 
projection map of jet schemes $u:J^jG \map J^{j-1}G$ for all $j$.
\color{black}

\begin{proposition}
	\label{diff}
	For any character $\Theta$ in $\bX_n(G)$, let the derivative at the identity
	with respect to our chosen coordinates 
	be $D\Theta = (A_0, \cdots, A_n)$ where $A_j \in \mb{Mat}_{1\times g} (R)$.
	\begin{enumerate}
		\item We have
		$$
		i^*\phi^*\Theta = \mfrak{f}^*(i^*\Theta)+ \gamma. \Psi_1,
		$$
		where $\gamma=\pi A_0$. 
		\item For $n\geq 2$, we have
		$$
		i^*(\phi^{\circ n})^*\Theta= (\mfrak{f}^{n-1})^* i^*\phi^*\Theta.
		$$
	\end{enumerate}
\end{proposition}
\begin{proof} See Proposition $6.3$ in \cite{BS_b}.
\end{proof}	
\subsection{Finiteness of the $R$-module $\bXp(A)$ and $\bH_{\d}(A)$}
Let $A$ be a $\pi$-formal abelian scheme of relative dimension $g$ over 
$\Spf R$.
For every $n$ we have the following short exact sequence of $\pi$-formal schemes
\begin{equation}
	\label{se}
	0 \map N^n \map J^nA \map A \map 0.
\end{equation}
Applying $\Hom(-, \hG)$ to the above short exact sequence gives us
\begin{equation}
	\label{sel}
	0 \map \bX_n (A) \map \Hom(N^n,\hG) \stk{\partial}{\map} \Ext(A,\hG). 
\end{equation}
Then by the theory of extensions of groups that admit rational sections (see 
\cite{Serre}, page 185, Theorem 7) we have 
$\Ext(A,\hG)\simeq H^1(A,\Ou_A) \simeq R^g$. 
Let $\bI_n(A) := \mb{image}(\partial)$. 
Note that since for all $n$, there are maps $\Hom(N^n,\hG) \stk{u^*}{\inj} \Hom(N^{n+1},\hG)$,
we have $\bI_n(A) \subset \bI_{n+1}(A)$. Define 
\begin{equation}
	\label{bigI}
	\bI(A):= \varinjlim \bI_n(A)
\end{equation}
and 
$$
h_i= \rk \bI_i(A) - \rk \bI_{i-1}(A)
$$
for all $i \geq 1$.
We define the {\it upper splitting number} to be the smallest number 
$\mup\geq 1$ such that $h_n=0$ for all $n \geq \mup$.  
Note that $\mup$ exists since $$\bI_0(A) \subset \bI_{1}(A) \subset \cdots \subset \Ext(A,\hG)\cong R^g.$$
We define the {\it lower splitting number} to be the unique $\mlow$ satisfying
$\bX_{\mlow}(A) \ne \{ 0\}$ and $\bX_{\mlow-1}(A) = \{0\}$. 
We say a $\d$-character $\Theta \in \bX_n(A)_K$ is {\it primitive} if 
$$
\Theta \notin u^* \bX_{n-1}(A)_K + \phi^* (\bX_{n-1}(A)_K)_{\phi}.
$$
\begin{theorem}
	\label{maxorder}
	For any abelian scheme $A$ of dimension $g$, $\bX_\infty(A)_K$ is freely
	$K\{\phi^*\}$-generated by $g$ $\d$-characters of order at most $g+1$. 
\end{theorem}
\begin{proof}
	See Theorem $1.1$ in \cite{BS_b} and Theorem B in \cite{bui95}.
\end{proof}

Now define $$\bXp(A):= \varinjlim \bX_n(A)/\phi^*\bX_{n-1}(A)_{\phi}.$$

\begin{corollary}
	\label{xprim}
	We have dimension of $\bXp(A)_K$ as a $K$-vector space to be $g$, and
	$$\bXp(A)_K\simeq\bX_{\mup}(A)_K/\phi^* (\bX_{\mup-1}(A)_K)_{\phi} .$$
\end{corollary}
\begin{proof}
	See Corollary $7.7$ in \cite{BS_b}.
\end{proof}

\begin{corollary}
	\label{mell}
	If $g=1$, then $\mlow= \mup =: m$ and $\bXp(A)_K \simeq \bX_m(A)_K$.
\end{corollary}
\begin{proof}
	See Corollary $7.8$ in \cite{BS_b}.
\end{proof}

We will, in fact, show that the module of $\d$-characters is finitely generated as an $R\{\phi^{*}\}$-module. As a consequence, it gives an integral lattice 
of the $\d$-isocrystal. 

\begin{theorem}
\label{phi*_bij}
For $n\geq m_{u}$, we have the following commutative diagram
	$${
\xymatrix{
	\bX_{n}(A)/\bX_{n-1}(A)\ar[d]_-{i^{*}} \ar[r]^{\phi^{*}} &
\bX_{n+1}(A)/\bX_{n}(A)\ar[d]^{i^{*}}  \\
	\Hom(N^n,\hG)/\Hom(N^{n-1},\hG) \ar[r]^{\mathfrak{f}^{*}} &\Hom(N^{n+1},\hG)/\Hom(N^n,\hG) }}
	$$
\cblue
with all the maps as isomorphisms of $R$-modules. 
\color{black}
\end{theorem}
\proof From Proposition 6.4 in \cite{BS_b} we already have $i^{*}$ and $\phi^{*}$ are injective and $\mathfrak{f}^{*}$ is bijective. Consider the exact sequence of finite free $R$-modules
$$0 \longrightarrow \bX_{n}(A)\longrightarrow \Hom(N^n,\hG)\xrightarrow{\partial} \bI_{n}(A) \longrightarrow 0.$$
Since $\bI_{n}(A)$ is free, there is a section $s_{n}:\bI_{n}(A)\longrightarrow \Hom(N^n,\hG)$. Moreover, this section implies that the following exact sequence of $R$-modules
$${
	\xymatrix{
		0\ar[r]&\bX_{n}(A)/\bX_{n-1}(A) \ar[r]^{i^{*}} &\frac{\Hom(N^n,\hG)}{\Hom(N^{n-1},\hG)}\ar[r]^{\partial} &\bI_{n}A/\bI_{n-1}A\ar@/^1.2pc/[l]^{s}\ar[r]& 0 \\
}}
$$
splits. By Proposition 5.2 in \cite{BS_b} we have $\Hom(N^n,\hG)/\Hom(N^{n-1},\hG)$ is a free module of rank $g$. Hence $\bI_{n}A/\bI_{n-1}A$ is also free being a sub-module of a free module over a discrete valuation ring. On the other hand,  $$\mathrm{rk}\left(\bI_{n}A/\bI_{n-1}A\right)=\mathrm{rk}(\bI_{n}A)-\mathrm{rk} (\bI_{n-1}A)=h_{n}.$$
Since $h_{n}=0$ for all $n\geq m_{u}$, we have $\bI_{n}A/\bI_{n-1}A=0$ and $i^{*}$ is bijective. Now all the maps are bijective, that will make $\phi^{*}$ also bijective. 
\qed

Let us define a finite subset  $\mathbb{D}_{i}\subset \bX_{i}(A)$ such that the image of its elements in $$\dfrac{\bX_{i}(A)}{u^{*}\bX_{i-1}(A)+\phi^{*}(\bX_{i-1}(A))_{\phi}}$$ are distinct and form a minimal generating set. Since $R$ is a discrete valuation ring and the modules are finite, the cardinality of such a minimal generating set is well-defined.
Such a $\mathbb{D}_{i}$ is called a \emph{primitive generating set} of $\bX_{i}(A)$. 
Let us define the subset $$S_{n}(\mathbb{D}_{i})=\{{\phi^{*}}^{h}\Theta ~|~\Theta\in \mathbb{D}_{i}, 0\leq h\leq n-i \} .$$
\begin{lemma} The  $R$-module $\bX_{n}(A)$ can be generated by
	$$ S_{n}(\mathbb{D}_{1})\cup S_{n}(\mathbb{D}_{2})\cup \ldots \cup S_{n}(\mathbb{D}_{n}) .$$
\end{lemma}
\proof We will proceed by induction on $n$. For $n=1$ is clear from the definition. Let us denote $$ \mathbb{S}_{i}=S_{i}(\mathbb{D}_{1})\cup S_{i}(\mathbb{D}_{2})\cup \ldots \cup S_{i}(\mathbb{D}_{i}).$$ We assume that $\mathbb{S}_{n-1}$ generates $\bX_{n-1}(A)$ as $R$-module. Consider the short exact sequence of  $R$-modules
$$0\longrightarrow u^{*}\bX_{n-1}(A)+\phi^{*}(\bX_{n-1}(A))_{\phi}\longrightarrow\bX_{n}(A)\longrightarrow\dfrac{\bX_{n}(A)}{u^{*}\bX_{n-1}(A)+\phi^{*}(\bX_{n-1}(A))_{\phi}}\longrightarrow 0.$$
Invoking induction hypothesis, we have $\bX_{n}(A)$ is generated by $$\mathbb{D}_{n}\cup \mathbb S_{n-1}\cup \phi^{*}(\mathbb{S}_{n-1})\subset \mathbb{S}_{n}.$$
\qed


\begin{theorem}\label{X_prim} The $R$-module $\bX_{\mathrm{prim}}(A)$ is free of rank $g$.
\end{theorem}
\proof Since $R$ is a discrete valuation ring, it is enough to prove that $\bX_{\mathrm{prim}}(A)$ is finitely generated and $\pi$-torsion free.
We will firstly show $\bX_{\mathrm{prim}}(A)$ is a finitely generated 
$R$-module.

By Theorem \ref{phi*_bij}, for $n\geq m_{u}$, we have the following commutative diagram
$${
	\xymatrix{
		\bX_{n}(A)/\bX_{n-1}(A)\ar[d]_{p_{n}} \ar[r]_{\phi^{*}}^{\sim} &\bX_{n+1}(A)/\bX_{n}(A)\ar[d]^{p_{n+1}}  \\
		\dfrac{\bX_{n}(A)}{u^{*}\bX_{n-1}(A)+\phi^{*}(\bX_{n-1}(A))_{\phi}} \ar[r]^{\bar{\phi}^{*}} &\dfrac{\bX_{n+1}(A)}{u^{*}\bX_{n}(A)+\phi^{*}(\bX_{n}(A))_{\phi}}}}
$$
Since, the vertical maps are surjective and $\phi^{*}$ is bijective, we have $\bar{\phi}^{*}$ also surjective. But note that $\bar{\phi}^{*}$ is a zero map. Hence this implies $\dfrac{\bX_{n+1}(A)}{u^{*}\bX_{n}(A)+\phi^{*}(\bX_{n}(A))_{\phi}}$ is zero for $n\geq m_{u}$. 
Therefore by the above lemma, $\bX_{\infty}(A)$ is generated by $$\mathbb{D}_{1}\cup \mathbb{D}_{2}\cup \ldots \cup \mathbb{D}_{m_{u}}$$ as $R\{\phi^{*}\}$-module, hence its image generates $\bX_{\mathrm{prim}}(A)$ as $R$-module.

Now we will show that $\bX_{\mathrm{prim}}(A)$ is $\pi$-torsion free.
\cblue
Let $\bx_0$ be a system of local \'{e}tale coordinates of $A$ around the 
identity section. Then as described in $(4.4)$ of \cite{BS_b}, there is 
a naturally induced \'{e}tale coordinate system $\bx=(\bx_0,\dots , \bx_n)$
around the identity section of $J^nA$.
\color{black}

Suppose we have $\pi [f] =0$ in $\bX_{\mathrm{prim}}(A)$ where 
$f \in \bX_n(A)$ for
some $n$. Hence there exists a 
\cblue
$g \in \bX_{n-1}(A)$ such that $\pi f = \phi^* g$. 
Suppose $g$, when written with respect to the \'{e}tale
coordinate system $\bx$ as described above, is given by 
$$g(\bx)=\displaystyle{\sum_{I}}a_{I}\mathbf{x}^{I} \in \bX_{\infty}(A),$$ 
where $I$ runs over multi-indices and $a_I \in R$ that are $\pi$-adically
converging to $0$ as $I$ goes to infinity.
Hence we have
\begin{equation*}
	\pi f(\mathbf{x})=\phi^{*}g (\mathbf{x})= g(\phi(\mathbf{x}))=\displaystyle{\sum_{I}}a_{I}\mathbf{x}^{qI}+\pi h(\mathbf{x}),
\end{equation*}
for some $h \in \Ou(J^nA)$.
\color{black}
The above equation shows that each $a_{I}$ is divisible by $\pi$, hence 
$\frac{1}{\pi}g\in \bX_{\infty}(A)$. Since $\bX_{\infty}(A)$ is free and hence $\pi$-torsion free, we obtain $f=\phi^{*}(\frac{1}{\pi}g)$, which implies
that $[f] =0$ in $\bXp(A)$. Hence $\bX_{\mathrm{prim}}(A)$ is free and by Lemma \ref{xprim} the rank has to be $g$.\qed

Recall from Section $8$ of \cite{BS_b} that the following diagram of short exact sequences of $R$-modules commutes.
\begin{equation}
	\label{diag-crys-limit}
	\xymatrix{
		0 \ar[r] & \bXp(A)\ar[d]_\Upsilon \ar[r] & 
		\Hd(A) \ar[d]_\Phi \ar[r] &\bI(A) \ar@{^{(}->}[d] \ar[r] &  0 \\
		0 \ar[r] & \dualmod{\Lie (A)} \ar[r] & \Ext^\sharp(A,\hG)\ar[r] & \Ext(A,\hG)
		\ar[r] & 0 
	}
\end{equation}
Hence as a direct consequence we get:

\begin{theorem}
	\label{isocrys} Let $A$ be an abelian scheme of dimension $g$ over $R$. Then  $\Hd(A)$ is a free $R$-module with $g\leq \rk_R \Hd(A) \leq 2g$.
\end{theorem}

\begin{proof}[\cbl{\bf Proof of Theorem \ref{fin-gen}}]
 Since $ \mathbb{D}_{1}\cup\mathbb{D}_{2}\cup \ldots \cup \mathbb{D}_{m_{u}}$ generates $\bX_{\mathrm{prim}}(A)$ as $R$-module, we can choose a minimal generating subset and its cardinality has to be $g$, because $\bX_{\mathrm{prim}}(A)$ has rank $g$ and $R$ is a discrete valuation ring.
Hence $\bX_{\infty}(A)$ will be freely generated by $g$ $\d$-characters of order upto $m_{u}\leq g+1$ as in Theorem 7.6 in \cite{BS_b}.
\end{proof}

\section{Delta Isocrystal and Crystalline Cohomology of Elliptic Curves}
\label{last}
In this section, we prove the comparison isomorphism between the $\d$-isocrystal and the crystalline cohomology of an elliptic curve over $\QQ_{p}$ and recover the usual Hodge filtration in terms of $\d$-characters. Let $A$ be an elliptic curve over $R$ and fix an invariant differential 1-form $\omega$ of $A$. Recall from \eqref{sel} that we have the following exact sequence of $R$-modules:
$$0 \map \bX_n (A) \map \Hom(N^n,\hG) \stk{\partial}{\map} \Ext(A,\hG).$$
Further, recall from Section $9$ in \cite{BS_b}, that for an elliptic curve $A$, the character group $\bX_{\infty}(A)=R\{\phi^{*}\}\langle\Theta_{m}\rangle$
\cblue
and $\Psi_i$s defined as in (\ref{Psi_i}). 
\color{black}
Then we have  two possible cases:
\begin{itemize}
	\item[(i)] $A$ has a canonical lift iff $m=1$, then $i^{*}\Theta_{1}=\Psi_{1}$.
	\item[(ii)]Otherwise, we have $m=2$ and $i^{*}\Theta_{2}=\Psi_{2}-\lambda \Psi_{1}$, where $\lambda \in R$. In this case $\partial \Psi_{1}$ is nonzero and since $i^{*}\Theta_{2}$ lies in the kernel of $\partial$,  we get $\lambda=\partial \Psi_{2}/\partial \Psi_{1}$. 
\end{itemize}
Thus pulling back $\Theta_{1},\Theta_{2}$ by $\phi$ and $i$, we have 
\begin{align*}
	i^{*}\phi^{*}\Theta_{1}&=\mff^{*}i^{*}\Theta_{1}+\gamma \Psi_{1}=\Psi_{2}+\gamma \Psi_{1}\\
	i^{*}\phi^{*}\Theta_{2}&=\mff^{*}i^{*}\Theta_{2}+\gamma \Psi_{1}=\Psi_{3}-\phi(\lambda)\Psi_{2}+\gamma \Psi_{1}.
\end{align*} 

\subsection{Geometric Interpretation of the Arithmetic Picard-Fuchs Operator}
\label{apf}
The main aim of this subsection is to prove Proposition \ref{Lambdaform} 
that will play an important role in proving Theorem \ref{Iso_crys}.
Recall from Lemma 2.8 in \cite{bui95}, we have the morphism $\vp:\hG\longrightarrow A$, which can also be described as the composition of the following maps below:
\begin{align*}
 \hG \stk{{\frac{1}{\pi}\mathrm{exp}_{\mathcal{F}}}(\pi x)}
	{\longrightarrow} N^{1}A \stk{i}{\map} J^{1}A 
\stk{\phi}{\longrightarrow} A  
	\hspace{4cm}	 \\
	\nonumber  x \mapsto 1/\pi ~ \mathrm{exp}_{\mathcal{F}}(\pi x) \mapsto(0,1/\pi ~ \mathrm{exp}_{\mathcal{F}}(\pi x))\mapsto \mathrm{exp}_{\mathcal{F}}(\pi x) \hspace{2cm} 
\end{align*}
where $\mathcal{F}$ denotes the formal group associated to the group law of $A$
\cblue and $\exp_{\mathcal{F}}$ is the formal exponential map corresponding to 
the formal group law $\mathcal{F}$ with $x$ as a chosen coordinate system 
around the zero section of the additive group law.  \color{black}
Therefore as in Page $324$ in \cite{bui95}, this induces an injective pullback 
map
\begin{align*}
\vp^{*}:\bX_{\infty}(A)\longrightarrow \bX_{\infty}(\hG)=R\{\phi_{\hG}\} ~\mathrm{such~that}\\ 
	\nonumber	\vp^{*}(\bX_{n}(A))\subset \bX_{n}(\hG)=R+R\langle
\phi_{\hG} \rangle+\ldots+R\langle\phi^{n}_{\hG}\rangle.
\end{align*}

Let $B$ be a $\pi$-adically complete $R$ algebra which has a $\pi$-derivation
$\d$ that lifts the fixed derivation on $R$. For each $n$, by the universal 
property of 
Witt vectors (cf. Section 1 in \cite{bor11a}) the $\pi$-derivation induces the canonical map
$B \stk{\exp_\d}{\longrightarrow} W_n(B)$. Hence given a $B$-point of a 
$\pi$-formal scheme $X$ induces $\nabla: X(B) \map X(W_n(B))= J^nX(B)$.

Consider
the morphism $\phi \circ i : N^1A \map A$. Then by the universal property 
of jet spaces, Proposition \ref{univ}, this induces a unique map of 
prolongation 
sequences $N^{*+1}A \map J^*A$. In the case of $m=0$ in Theorem \ref{lateral},
it is easy to see
that the morphism between the prolongation sequences at each level is given by 
$\phi \circ i: N^{n+1}A \map J^nA$. Also by Theorem \ref{formalcase1} in the 
case $m=0$, for all $n \geq 0$ we have $N^{n+1}A \simeq J^n(N^1A)$.

In the case when $A$ is an elliptic curve over $\Spf R$, we
have an isomorphism $\vartheta^{-1}:\hG \map N^1A$ and hence induces an isomorphism
$J^n(\vartheta^{-1}): J^n\hG \simeq \bb{W}_n \map N^{n+1}A$.

\begin{lemma}
	\label{nuiso}
	Let $\Psi: N^nA \map \hG$ be given by $\Psi =  b_1 \Psi_1 +
	\cdots + b_n\Psi_n$ where $b_i \in R$ for all $i=1, \dots , n$. Then 
	$$\Psi \circ J^{n-1}(\vartheta^{-1}) = b_1 \mathbbm{1}_{\hG} + 
	b_2 \phi_{\hG} +\cdots + b_n \phi^{n-1}_{\hG}.$$
\end{lemma}
\begin{proof}
It is enough to prove the result for $\Psi = \Psi_i$ for all $i$.
Recall $\Psi_i = \vartheta \circ 
	\fra^{\circ (i-1)}$. By Theorem \ref{formalcase1}, we have for all $j \leq 
	n-1$
	$$\vartheta^{-1} \circ \phi_{\hG}^{\circ j} = \fra^{\circ j}
	\circ J^{n-1}(\vartheta^{-1}).$$ 
Then we have 
\beqar
	\Psi \circ J^{n-1}(\vartheta^{-1}) &=& \vartheta \circ 
\fra^{\circ (i-1)} \circ J^{n-1}(\vartheta^{-1}) \\
	&=& \vartheta \circ \vartheta^{-1} \circ \phi_{\hG}^{(i-1)} \\
	&=& \phi_{\hG}^{(i-1)}
\eeqar
and this proves our result.
\end{proof}

Let $\Theta \in \bX_r(A)$ be a delta character of order $r$ of $A$. Then 
$\Theta: J^rA \map \hG$ is a morphism of $\pi$-formal group schemes. Hence
again by the universal property of jet spaces we have for each $n \geq 0$
the following compatible system of morphisms
\begin{align}
\label{Theta-diagram}
\xymatrix{
	J^{n+r}\hG \ar[d] \ar[r]^-{J^{n+r}(\vartheta^{-1})} & N^{n+r+1}A \ar[d] \ar[r]^{\phi \circ i} & J^{n+r}A \ar[d]\ar[r]^{J^n(\Theta)} & J^n\hG \ar[d] \\
	J^r\hG \ar[r]^{J^r(\vartheta^{-1})} \ar[d] & N^{r+1}A \ar[d] 
\ar[r]^{\phi \circ i} & J^{r}A \ar[d] \ar[r]^{\Theta} & \hG \\
	\hG \ar[r]^{\vartheta^{-1}} & N^1A \ar[r]^{\phi \circ i} & A & 
}	
\end{align}

Consider the morphism $\Lambda_\theta: J^r \hG \map \hG$ given by the 
following composition
\begin{align}
	\label{Lambda}
	\Lambda_\Theta:= \Theta \circ (\phi \circ i) \circ 
J^r(\vartheta^{-1}).
\end{align}
Note that $\Lambda_\Theta$ is a delta character of $\hG$ of order $r$, in other
words, $\Lambda_\Theta \in \bX_r(\hG)$. This is known as {\it{arithmetic Picard-Fuchs operator}}. It is also called the symbol of $\Theta$ in 
\cite{buium-miller}.   
Hence for all $i= 0, \dots, r$ there exist $b_i \in R$ such that 
\begin{align}
	\label{Lambda-form}
	\Lambda_\Theta &=  b_r \phi^r_{\hG} + \cdots + b_1 \phi_{\hG}+ b_0
\end{align}

For any $\pi$-formal scheme $G$, by Theorem \ref{formalcase1} we have $N^{r+1}G
\simeq J^r(N^1G)$. Hence for 
any $R$-point of $N^1G$, we can consider the canonical lift $\nabla: N^1G(R) 
\map J^r(N^1G)(R) \simeq N^{r+1}G (R)$.
Then evaluating diagram (\ref{Theta-diagram}) on $R$-points we obtain 

\begin{align}
	\label{Lambda-R-dia}
\xymatrix{
	J^r\hG(R) \ar[r]^{J^r(\vartheta^{-1})} \ar[d] & N^{r+1}A(R) \ar[d] 
	\ar[r]^{\phi \circ i} & J^{r}A(R) \ar[d] \ar[r]^{\Theta} & \hG(R) \\
	\hG(R) \ar@/^1pc/[u]^\nabla \ar[r]^{\vartheta^{-1}} & N^1A(R) 
	\ar@/^1pc/[u]^\nabla \ar[r]^{\phi \circ i} & A(R) 
	\ar@/^1pc/[u]^\nabla & 
}	
\end{align}

When $A$ does not have CL, then $r=2$ and the $\phi$-linear endomorphism of $R$
in Page 325 of \cite{bui95} is obtained by the following composition
$$
\Theta \circ \nabla \circ(\phi \circ i)  \circ \vartheta^{-1}
$$
Then by (\ref{Lambda-R-dia}) the above map is the same as 
$$
\Theta \circ (\phi \circ i) \circ J^2(\vartheta^{-1}) \circ \nabla = 
\Lambda_\Theta \circ \nabla
$$
The following result relates $\Lambda_\Theta$ with $i^*\Theta$ geometrically:
\begin{proposition} \label{Lambdaform}
	Let $A$ be an ellitpic curve over $R$, and $\Theta \in \bX_r(A)$ such
	that $i^* \Theta = a_r \Psi_r + \cdots + a_1 \Psi_2 +  + a_1 \Psi_1$. 
	Then
$$
	\Lambda_\Theta = \phi(a_r) \phi^r_{\hG} + \cdots + \phi(a_1) \phi_{\hG}
	+ \gamma 
$$
	where $\gamma \in R$ is as given in Proposition \ref{diff}.
\end{proposition}

\begin{proof}
	By Proposition \ref{diff}(1) we have 
	$$
	(\phi \circ i)^* \Theta = \phi(a_r) \Psi_{r+1} + \cdots + \phi(a_1) 
	\Psi_2 + \gamma \Psi_1
	$$
	Then $\Lambda_\Theta = (\phi \circ i)^* \Theta \circ 
	J^r(\vartheta^{-1})$ and the result follows from applying Lemma 
	\ref{nuiso} to the above.
\end{proof}

\subsection{Proof of Theorem \ref{Iso_crys}}
\label{mainthm}\
\cblue
Let $\mr{\Iso}(H) = (H,F,(H \supset V \supset \{0\}))$ be a filtered isocrystal over $\Q_p$ 
where $H$ is a two dimensional vector space over $\Q_p$, $F:V \map V$ is a 
semilinear (in fact, in this case $F$ is linear since $H$ is a $\Q_p$-vector 
space)
operator which is a bijection and $V$ is a one-dimensional $\Q_p$-subspace 
of $H$. Let $p_F(t) \in \Q_p[t]$ be the degree two
characteristic polynomial of $F$.

\begin{proposition}
\label{filiso}
Let $\mr{\Iso}(H)=(H,F,(H \supset V \supset \{0\}))$ and 
$\mr{\Iso}(H')=
\newline
(H',F',(H' \supset V' \supset \{0\}))$ 
be filtered isocrystals over $\Q_p$ such that 
\begin{enumerate}
\item $\dim_{\Q_p} H = \dim_{\Q_p} H' = 2$
\item $F(V) \ne V$ and $F'(V') \ne V'$ and
\item $p_F(t) = p_{F'}(t)=: p(t)$. 
\end{enumerate}
Then $\mr{\Iso}(H) \simeq \mr{\Iso}(H')$ in the category of filtered isocrystals
over $\Q_p$.
\end{proposition}

\begin{proof}

Let $p(t) = t^2 - at -b$ for some $a,b \in \Q_p$ and choose any non-zero vector
 $v \in V$. Since $F(V) \ne V$, the set $\{v, F(v)\}$ forms a $\Q_p$-basis of 
$H$. Then we have 
$$
F^{\circ 2}(v) = a F(v) + b v.
$$
Similarly, for any non-zero vector $w \in V'$ the set $\{w,F'(w)\}$ is a 
$\Q_p$-basis for $H'$ and we have
$$
F'^{\circ 2}(w) = a F'(w) + b w.
$$
Define the $\Q_p$-linear map $\Phi: H \map H'$ given by $\Phi(v):= w$ and 
$\Phi(F(v)) := F'(w)$. Then $\Phi$ is an isomorphism of $\Q_p$-vector spaces
that satisfies
$$
\xymatrix{
H \ar[d]_F \ar[r]^\Phi & H'\ar[d]^{F'}\\
H \ar[r]^\Phi & H'
}
$$
such that $\Phi(V) = V'$. Hence $\Phi:\mathrm{\Iso}(H) \map \mathrm{\Iso}(H')$
is the required isomorphism of filtered isocrystals and this completes our
proof.




\end{proof}

\color{black}

We have the following two cases depending on whether $A$ admits a lift of Frobenius or not. In both cases, we will first show that $\mathrm{\Iso}(\Hd(A))$ is a weakly admissible isocrystal over $\QQ_{p}$ and then prove part (1) and part (2) of the theorem respectively.

(1) \textbf{\underline{Non-CL case}:} Let $A$ be a non-CL elliptic curve over $\ZZ_{p}$. Then by Theorem 9.7(b) in \cite{BS_b}, we have $\bX_{\infty}(A)= R\{\phi^{*}\} \langle \Theta_{2}\rangle$ and $\Hd(A)=R\langle\Psi_{1},\Psi_{2}\rangle.$ 
Recall from Section $9$ in \cite{BS_b} that the matrix of $\mff^{*}$ with respect to the basis $\{\Psi_{1},\Psi_{2}\}$ is given by
	\[
	[\mathfrak{f}^{*}]=
	\begin{bmatrix}
		0  & -\gamma  \\
		1 & \lambda & 
		
	\end{bmatrix}.
	\]
Note that $\phi$ is identity on $\Z_{p}$ and hence $\mff^{*}$ becomes linear. 
Note that, since $i^{*}\Theta_{2}=\Psi_{2}-\lambda \Psi_{1}$ and 
$\phi(\lambda)=\lambda$ because $\lambda \in \Z_p$, by Proposition \ref{Lambdaform}, we obtain that the arithmetic Picard-Fuchs operator of $\Theta_{2}$ as
	\begin{align*}
	\Lambda_{\Theta_{2}}=\phi_{\hG}^{2}-\lambda \phi_{\hG} +\gamma \in
\bX_\infty(\hG). 
	\end{align*}	
	On the other hand, from Theorem $1.10$ in \cite{bui97}, we have 
$$\Lambda_{\Theta_{2}}=\phi_{\hG}^{2}-a_{p}\phi_{\hG}+p\in \bX_{\infty}(\hG),$$
	where $a_{p}=p+1-\#A_{0}(\FF_{p})$.
Therefore under this identification, it follows that $\lambda=a_{p}$ and 
$\gamma=p$. Hence the characteristic polynomial of $\fra^*$ is given by
$$
p_{\fra^*}(t) = t^2 - \lam t + \gamma = t^2 - a_p t + p.
$$

Combining Theorem 9.7 in \cite{BS_b}, this proves that $\mathrm{\Iso}(\Hd(A))$ is weakly admissible filtered isocrystal.





\cblu{
Consider the filtered isocrystal 
$$
\mathrm{\Iso}(\bH_\d(A)_{\Q_p}) = (\bH_\d(A)_{\Q_p}, \fra^*, 
\bH_\d(A)_{\Q_p}^\bullet)
$$
where $\bH_\d(A)_{\Q_p}^\bullet$ is the filtration given by 
$\bH_\d(A)_{\Q_p} \supset \bXp(A)_{\Q_p} \supset \{0\}$.
Since $A$ is non-CL, we have $\fra^*(\bXp(A)_{\Q_p}) \ne \bXp(A)_{\Q_p}$ and
the characteristic polynomial of $\fra^*$ is $p_{\fra^*}(t) = t^2 - a_p t +p$.

On the other hand, consider the filtered isocrystal of the first crystalline
cohomology $\mathrm{\Iso}(\Hcr(A)_{\Q_p})= (\Hcr(A)_{\Q_p}, \Fc,
\Hcr(A)^\bullet)$ where $\Hcr(A)^\bullet$ is the Hodge filtration $\Hcr(A) 
\supset H^0(A,\Omega_A)\supset \{0\}$ and $\Fc$ is the crystalline Frobenius
operator on $\Hcr(A)$.
Since $A$ is a non-CL elliptic curve over $\Z_p$, by Theorem $3.15$ of 
\cite{BO},
$\Fc (H^0(A,\Omega_A)_{\Q_p}) \ne H^0(A,\Omega_A)_{\Q_p}$ and the characteristic
polynomial of $\Fc$ is $p_{\Fc}(t) = t^2 - a_p t + p$. 

Hence by Proposition \ref{filiso} applied to $\mathrm{\Iso}(\bH_\d(A)_{\Q_p})$
and $\mathrm{\Iso}(\Hcr(A)_{\Q_p})$ we obtain our required isomorphism of 
filtered isocrystals.
}





\vspace{.2cm}

\noindent (2) \textbf{\underline{CL case}:} 
\cbl{
Let $\mu :A \map A$ denote the canonical lift of Frobenius on $A$. Then $\mu$
can be written as $\mu(x) = x^p + p f(x)$ where $x$ is a local \'{e}tale 
coordinate around the identity section of $A$ and $f(x)$ is a restricted
formal power series in $x$ with coefficients in $\Z_p$. 

Note that we have the following injection of rings
$$
\End (A) \inj \End_{\Z_p} (\Lie (A)),
$$
where an endomorphism is sent to multiplication on the Lie algebra of $A$ 
by the derivative of the endomorphism at the identity section.
In particular for our given $\mu$, the derivative operator is given by
multiplication of
$\beta:= D\mu (0) = p f'(0)$. Note that $\beta \ne 0$ since $\mu$ is a 
non-trivial endomorphism of $A$ and we have $\rm{val}_p(\beta) \geq 1$.

Since $A$ is a CL elliptic curve, the crystalline Frobenius $\Fc$ on $\Hcr(A)$ 
preserves the Hodge filtration and is induced by pulling back via $\mu$.
Hence $\Fc: H^0(A,\Omega_A) \map H^0(A,\Omega_A) $ is given 
by $\Fc(\omega) = \beta \omega$ for all $\omega \in 
H^0(A, \Omega_A)$. Hence $\beta$ is an eigen value for the operator
$\Fc$ and therefore is one of the roots of its characteristic polynomial 
$p_{\Fc}(t) = t^2 -a_p t +p$. Let $\alpha$ be the other root of $p_{\Fc}(t)$.
Then $\alpha \beta = p$ which implies $\rm{val}_p(\beta) = 1$ and 
$\rm{val}_p(\alpha) =0$, that is $\alpha \in \Z_p^\times$.

By Theorem  9.7(b) in \cite{BS_b}, we have $\bX_{\infty}(A)= R\{\phi^{*}\} \langle \Theta_{1}\rangle$ and $\Hd(A)=R\langle\Psi_{1}\rangle.$ Therefore the semilinear operator act as 
$$\mff^{*}(\Psi_{1})=-\gamma\Psi_{1}.$$ 
Here we have $i^*\Theta_{1}=\Psi_{1}$, hence by Proposition \ref{Lambdaform}, we obtain the arithmetic Picard-Fuchs operator of $\Theta_{1}$ as
	\begin{align*}
		\Lambda_{\Theta_{1}}=\phi_{\hG}+\gamma\in \bX_\infty(\hG). 
	\end{align*}	
From Theorem $1.10$ in \cite{bui97}, we have 
	$$\Lambda_{\Theta_{1}}=\phi_{\hG}-\beta\in \bX_{\infty}(\hG),$$
	where $\beta$ is the nonunit root of the polynomial $t^{2}-a_{p}t+p$. Therefore we have $\gamma=-\beta$. 
Hence by Theorem 9.7 of \cite{BS_b} we have that
$(\Hd(A),\mff^{*})$ is a weakly admissible filtered isocrystal as it is one dimensional. 

Let $v$ be a basis vector for $\Hd(A)$ and $w$ a basis vector 
for $H^0(A,\Omega_A)$ over $\Q_p$. Define the $\Q_p$-linear map $\Phi: \Hd(A)
\map H^0(A,\Omega_A)$ given by $\Phi(v) = w$. Then $\Phi$ is the required
isomorphism of isocrystals and this completes the proof.
}


\cblu{
\section{Appendix: The functor of points approach}
\label{FuncPoin}


In this section we will reprove Theorem \ref{Niso1} and Theorem \ref{affjet1}
using the functor of points approach.
The following elegant method has been pointed out by the anonymous referee to
whom we are greatly indebted.

As in the proof of Theorem \ref{Niso1}, it is sufficient to show the result
in the case when $X = \A^1 = \Spec A$ where $A = R[x]$ is the polynomial 
over $R$. Hence we will reprove Theorem \ref{coord4} here.

Then $N^{[m]n}X = \Spec N_{[m]n}A$ and $\Delta$ be the unique $\pi$-derivation
associated to the lateral Frobenius $\fra$ and satisfies 
\begin{align}
\label{fraAp}
\fra(a) = a^q + \pi \Delta(a)
\end{align}
for all $a \in N_{[m]n}A$ and $n \geq 1$.

Let $C$ be a $\pi$-torsion free $\Ou$-algebra. 
For each $m,n\geq0$, consider the {\it two dimensional ghost map}
$w_{m,n}: W_m(W_n(C))\longrightarrow \Pi_m(\Pi_n(C))$ given as the composition
$$
w_{m,n}:W_m(W_n(C)) \stk{W_m(w)}{\longrightarrow}  W_m(\Pi_n (C)) 
\stk{w}{\longrightarrow} \Pi_m \Pi_n (C).
$$
where $w$ is the usual ghost map. An element $x \in W_m(W_n(C))$ can 
be written as 
\begin{align}
\label{2dimx}
x = \begin{pmatrix}
x_{00} & x_{10} &\ldots &x_{m0}\\
x_{01} & x_{11} &\ldots &x_{m1}\\
\vdots &\ddots& &\vdots\\
x_{0n} & x_{1n} &\ldots & x_{mn}
\end{pmatrix}
\end{align}
where $x_{ij} \in C$ for all $i = 0, \dots , m$ and $j = 0,\dots ,n$. Then 
$w_{m,n}(x)$ is given by the following
$$
\begin{pmatrix}
x_{00} & {x_{00}}^q+\pi x_{10} &\ldots & x_{00}^{q^m} +\pi {x_{10}}^{q^{m-1}}+\ldots+ \pi^{m} x_{m0} \\
{x_{00}}^q+ \pi x_{01} & ({x_{00}}^q+ \pi x_{01})^q +\pi ({x_{10}}^q+\pi x_{11})&\ldots & ({x_{00}}^q+ \pi x_{01})^{q^m}+\ldots+ \pi^{m} (x_{m0}^{q}+\pi x_{m1}) \\
\vdots &\ddots& \vdots & \vdots
\end{pmatrix}.
$$

Define the \textit{hook map} $h_{m,n}:\Pi_m(\Pi_n(C))\longrightarrow 
\Pi_{m+n}(C)$ given by the concatenation of the top row with the right column 
of a matrix as follows:
\begin{align*}
h_{m,n}\left\langle\begin{matrix}
z_{00} & z_{10} &\ldots &z_{m0}\\
z_{01} & z_{11} &\ldots &z_{m1}\\
\vdots &\ddots& &\vdots\\
z_{0n} & z_{1n} &\ldots &z_{mn}
\end{matrix}\right\rangle :=\langle{z_{00},z_{10},\ldots, z_{m0},z_{m1},\ldots, z_{mn}}\rangle.
\end{align*}

\begin{proposition}\label{hook-prop}
 \begin{itemize} 
\item[(i)] For each $m,n\geq0$, and a $\pi$-torsion free $\Ou$-algebra $C$, there exists a unique map $r_{m,n}: W_{m}(W_{n}(C))\longrightarrow W_{m+n}(C)$ such that the following diagram commutes
$$
\xymatrix{
 W_{m}(W_{n}(C)) \ar[d]_{w_{m,n}}\ar[r]^{r_{m,n}} &  W_{m+n}(C)\ar[d]_{w_{m+n}}\\
\Pi_m(\Pi_n(C))\ar[r]^{h_{m,n}}  & \Pi_{m+n}(C).
}$$
\item[(ii)] There is a unique functorial family of $\Ou$-algebra homomorphisms $r_{m,n}:W_{m}(W_{n}(C))\longrightarrow W_{m+n}(C)$, where $C$ ranges over all $\Ou$-algebras. Moreover each $r_{m,n}$ is a retraction of the comonad structure map $\Delta: W_{m+n}(C)\longrightarrow W_{m}(W_{n}(C))$.
\end{itemize}
\end{proposition}

Before we prove Proposition \ref{hook-prop}, we will need the following 
results.

\begin{lemma}\label{n=1} Proposition $\ref{hook-prop}(i)$ is true for $n=1$
and $m \geq 0$.
\end{lemma}
\begin{proof} 

We will first show the existence of the map $r_{m,1}$.
Let $x \in W_m((W_1(C))$. Then its image under the two dimensional ghost map 
$w_{m,n}$ is given by
\begin{align*}
 \left\langle\begin{matrix}
x_{00} & {x_{00}}^q+\pi x_{10} &\ldots & x_{00}^{q^m} +\pi {x_{10}}^{q^{m-1}}+\ldots+ \pi^{m} x_{m0} \\
{x_{00}}^q+ \pi x_{01} & ({x_{00}}^q+ \pi x_{01})^q +\pi ({x_{10}}^q+\pi x_{11})&\ldots & ({x_{00}}^q+ \pi x_{01})^{q^m}+\ldots+ \pi^{m} (x_{m0}^{q}+\pi x_{m1}) 
\end{matrix} \right \rangle.
\end{align*}
where $x$ is represented as in (\ref{2dimx}). 
Composing with the hook map, $h_{m,n}(w_{m,n}(x))$ is given by
\begin{align*}
\langle {x_{00}, {x_{00}}^q+\pi x_{10} ,\ldots, x_{00}^{q^m} +\pi {x_{10}}^{q^{m-1}}+\ldots+ \pi^{m} x_{m0},({x_{00}}^q+ \pi x_{01})^{q^m}+\ldots+ \pi^{m} (x_{m0}^{q}+\pi x_{m1})}\rangle.
\end{align*}
We need to show that this ghost vector is the image of a (necessarily unique) Witt vector in $W_{m+1}(C)$.

We need to find an element $x'=(x_0,\ldots, x_m, z)$ in $W_{m+1}(C)$ satisfying $w_{m+1}(x')=h_{m,n}(w_{m,n}(x))$. Since $C$ is $\pi$-torsion free, solving the 
equations obtained from comparing the ghost coordinates we obtain that
$x_{i}=x_{i0}$ for $0\leq i\leq m$, and $z$ satisfies 
\begin{align*} 
x_{00}^{q^{m+1}} +\pi {x_{10}}^{q^m}+\ldots+ \pi^{m} x_{m1}^q+\pi^{m+1}z &=
({x_{00}}^q+ \pi x_{01})^{q^m}+\pi({x_{10}}^q+ \pi x_{01})^{q^{m-1}}\\
& \hspace{1cm} \cdots+ \pi^{m} (x_{m0}^{q}+\pi x_{m0})
\end{align*}

Hence we have
\begin{align}
\label{congruence}
\pi^{m+1}z &=[({x_{00}}^q+ \pi x_{01})^{q^m}-x_{00}^{q^{m+1}} ]+\pi[({x_{10}}^q+ \pi x_{01})^{q^{m-1}}-{x_{10}}^{q^m}] + \\
& \hspace{3cm} \cdots+ \pi^{m} [(x_{m0}^{q}+\pi x_{m0})-x_{m1}^q] \nonumber
\end{align}

Note that since $\pi$ divides $q$, the following congruences holds:
\begin{align}
({x_{i0}}^q+ \pi x_{i1})^{q^{m-i}}\equiv x_{i0}^{q^{m-i+1}}, \mb{ for all}
\mod \pi^{m-i+1},~ (0\leq i\leq m)
\end{align}

Hence the right hand side of \eqref{congruence} is divisible by $\pi^{m+1}$ 
and therefore there is a unique 
$z\in C$ satisfying the above identity. 
Hence define $$r_{m,1}(x):=(x_{00}, x_{10},\ldots, x_{m0},z)$$ and we are done.

Since $C$ is $\pi$-torsion free, the element $z$ is uniquely determined and 
therefore implies the map $r_{m,1}$ is uniquely determined.
\end{proof}

\textbf{Proof of Proposition \ref{hook-prop}}
($i$) Uniqueness follows from the ghost maps being injective on the $\pi$-torsion free ring $C$.
For existence, we use induction on $n$. The case $n=0$ is immediate. So assume $n\geq 1$.  Then consider the commutative diagram
$$
\xymatrix{
W_{m}W_{n}(C) \ar[d]_{w_{m,n}} \ar[r]^-{\Delta} & W_{m} W_{n-1}W_1(C) \ar[d] \ar[r]^{r_{m,n-1}} &  W_{m+n-1}(W_1(C))\ar[d] \ar[rr]^{r_{m+n-1,1}} && W_{m+n}(C) \ar[d]_{w_{m+n}} \\
\Pi_m\Pi_n(C)\ar[r]^-{\Delta}  & \Pi_{m} \Pi_{n-1} \Pi_1 (C) \ar[r]^{h_{m,n-1}}  & \Pi_{m+n-1}(\Pi_{1}(C))\ar[rr]^{h_{m+n-1,1}}  && \Pi_{m+n}(C)
}
$$
By induction hypothesis the map $r_{m, n-1}$ exists for the $\Ou$-algebra 
$W_1(C)$, instead of $C$.  The map $r_{m+n-1,1}$ exists by Lemma \ref{n=1}. Therefore the image of $W_m W_n (C)$ along the lower left route is contained in the image of the ghost map $w_{m+n}$; that is what we needed to prove. In fact, $r_{m,n}$ can be expressed in closed form as a compostion $r_{m+n-1,1} \circ r_{m,n-1}\circ \Delta$.

($ii$) Since the ring representing the functor $W_{m+n}$ is $\pi$-torsion free, the maps $r_{m,n}$ prolong to all $\Ou$-algebras follows from the $\pi$-torsion free statement. To prove that $r_{m,n}$ is a retraction of the comonad structure map, it is enough to check on ghost components. But this is true because the comonad structure map sends a ghost vector $(y_k)_{0\leq k\leq m+n}$ to the two dimensional ghost vector $((x_{ij})_{0\leq j\leq n})_{0\leq i\leq m}$ where $x_{ij}=y_{i+j}$.\qed

\begin{proposition}\label{hook-ver} The following diagram of short exact sequences is commutative
$$
\xymatrix{
0\ar[r]& W_{n-1}(C) \ar[dd]_{id} \ar[r]^{V^{m+1}} & W_{m+1} W_{n-1}(C) \ar[dd]^{r_{m+1,n-1}} \ar[r] &  W_{m}(W_{n-1}(C))\ar[d]^{r_{m,n-1}}  \ar[r]& 0\\
& & & W_{m+n-1}(C)\ar[d] &\\
0\ar[r] &W_{n-1}(C)\ar[r]^{V^{m+1}}  & W_{m+n}(C) \ar[r]  & W_{m}(C))\ar[r] & 0.
}
$$
\end{proposition}
\begin{proof} This can be checked using ghost components. The commutativity of the right square follows from the definition of the hook maps. We only need to 
check for the left square, which is the compatibility of the hook map $h_{m,n}$ with the ghost Verschibung $V_{w}$, which also holds as below 
$$h_{m+1,n-1}\langle V_{w}^{m+1}\langle{x_0,\ldots, x_{n-1}}\rangle\rangle=\langle 0,\dots,0, \pi^{n}x_{0},\ldots,\pi^{n}x_{n-1}\rangle=V_{w}^{m+1}\langle x_{0},\ldots,x_{n-1}\rangle.$$
\end{proof}
\subsection{The map $\alpha$} Let $(X,P)$ be a pointed affine scheme over $R$.
By slight abuse of notations, we again write $r_{m,n}$ for the induced map in jet spaces  $$r_{m,n}: J^{n}J^{m} X\longrightarrow J^{m+n}X.$$

\begin{corollary}\label{hook-jet} The diagram below commutes
$$
\xymatrix{
 J^{n-1} J^{m+1}X \ar[dd]^{r_{m+1,n-1}} \ar[r] &  J^{n-1}J^{m}X\ar[d]^{r_{m,n-1}}  \\
& J^{m+n-1}X\ar[d]^{u} \\
 J^{m+n}X \ar[r]^{u}  & J^{m}X.
}
$$
\end{corollary}
\begin{proof}
The result follows directly from Proposition \ref{hook-ver}.
\end{proof}

\begin{proposition} Let $(X,P)$ be a pointed affine scheme over $R$. Then there exists a canonical map $\alpha:J^{n-1}(N^{[m]1}X) \longrightarrow N^{\mn}X$ such that the following diagram commutes

$$
\xymatrix{
J^{n-1}(N^{[m]1}X) \ar@{-->}[rd]^{\alpha}\ar[rr] \ar[dd] && J^{n-1}J^{m+1}X \ar[rd]^{r_{m+1,n-1}} \ar[dd]^{J^{n-1}u}\\
& N^{\mn}X\ar@/^/[rr] \ar[dd]&& J^{m+n}X\ar[dd]^{u}\\
J^{n-1}(\Spec R) \ar@<1ex>[dr]^{\sim}_{u} \ar[rr]^{J^{n-1}P^m}  && J^{n-1}J^{m}X\ar[rd]^{u\circ r_{m,n-1}}\\
& \Spec R\ar[rr]^{P^{m}} && J^{m}X.
}
$$
\end{proposition}

\begin{proof} By Corollary \ref{hook-jet}, we have that the right square of the diagram commutes. Since $R$ is a $\delta$-ring, the marked point $P: \Spec R\longrightarrow X$ gives a marked point $P^m: \Spec R\longrightarrow J^m X$. We know by definition that $N^{\mn}X$ is the fibre of $u: J^{m+n}X\longrightarrow J^mX$ at the marked point $P^m$, which is the front square. In particular, invoking $n=1$ in the front square and applying the $J^{n-1}$ functor, we obtain the 
commuting square in the back.
The commutativity of the bottom square is clear. Then by the universal property of the fibre product for $N^{\mn}X$ in the front square, there exists a map $\alpha: J^{n-1}(N^{[m]1}X)\longrightarrow N^{\mn}X$ such that the entire diagram commutes. 
\end{proof}

\begin{theorem} 
\label{alpha0}
Let $(X,P)$ be a pointed affine scheme over $R$. Then the map $\alpha:J^{n-1}(N^{[m]1}X) \longrightarrow N^{\mn}X$ is an isomorphism.
\end{theorem}
\begin{proof} Note that it is enough to prove the statement for $(X,P)=(\mathbb{A}^{1},0)$. The general case follows from the case of $\mathbb{A}^1$ by a formal argument. Indeed, one can express $(X,P)\xrightarrow{\sim} \lim_{i} (\mathbb{A}^1, 0)$ as an equalizer in the category of pointed affine schemes over $R$.

Now, since the functors $J^{n-1}, N^{[m]1}$, and $N^{\mn}$ preserve limits, we have the canonical maps below are isomorphisms
$$J^{n-1}X\xrightarrow{\sim} \lim_{i} J^{n-1}\mathbb{A}^1, ~ N^{[m]1}X\xrightarrow{\sim} \lim_{i} N^{[m]1}\mathbb{A}^1, ~N^{\mn}X\xrightarrow{\sim} \lim_{i} N^{\mn}\mathbb{A}^1.$$ 
Since $X$ is affine implies that $N^{[m]1}X$ is affine.
Hence we obtain the isomorphism
$J^{n-1}(N^{[m]1}X)\xrightarrow{\sim} \lim_{i} J^{n-1}(N^{[m]1}\mathbb{A}^1)$.
Note that the maps $\alpha$ are functorial as $X$ varies, and since it is an isomorphism for $\mathbb{A}^1$, we obtain that $\alpha$ are isomorphisms for every $X$. 

In the case when $(X,P)=(\mathbb{A}^{1},0)$, by Lemma \ref{hook-ver} it follows that at the level of $C$-points $\alpha$ is the identity map of $W_{n-1}(C)$ for any $\pi$-torsion free $\Ou$-algebra $C$. This implies that $\alpha$ is identity and we are done.
\end{proof}

\begin{theorem} 
\label{alpha1}
Let $(X,P)= (\bb{A}^1,0)$ be the affine line over $S$ with 
the marked point denoted as the origin $0$. For all
$n$, the map
 $\alpha:J^{n-1}(N^{[m]1}X) \longrightarrow N^{\mn}X$ satisfies 
$$ 
\xymatrix{
 J^{n-1}N^{[m]1}X \ar[d]^{u} \ar[r]^{\alpha} & N^{\mn} X \ar[d]^{u}  \\
 J^{n-2}N^{[m]1}X \ar[r]^{\alpha}  & N^{[m]n-1}X
}~~
\xymatrix{
 J^{n-1}N^{[m]1}X \ar[d]^{\phi} \ar[r]^{\alpha} & N^{\mn} X \ar[d]^{\mff}  \\
 J^{n-2}N^{[m]1}X \ar[r]^{\alpha}  & N^{[m]n-1}X.
}
$$

\end{theorem}
\begin{proof} 

It is enough to check the above on the ghost side. 
Then for the left diagram above, we need to show that the following map of $n$ variables
\begin{align}\label{alpha}
\alpha_{w,n}:\left \langle\begin{matrix}
0 & 0 &\ldots &x_{0}\\
0 & 0 &\ldots &x_{1}\\
\vdots &\ddots&& \vdots\\
0 & 0 &\ldots &x_{n-1}
\end{matrix}\right \rangle \mapsto \langle 0,\ldots,0,x_{0},\ldots,x_{n-1}\rangle
\end{align}
is compatible as we truncate from $n$ to $n-1$, which is true.

Now observe that 
$$\phi_{w} \left \langle \begin{matrix}
0 & 0 &\ldots &x_{0}\\
0 & 0 &\ldots &x_{1}\\
\vdots &\ddots& &\vdots\\
0 & 0 &\ldots &x_{n-1}
\end{matrix}\right \rangle=
\left \langle\begin{matrix}
0 & 0 &\ldots &x_{1}\\
0 & 0 &\ldots &x_{2}\\
\vdots &\ddots& &\vdots\\
0 & 0 &\ldots &x_{n-1}
\end{matrix}\right \rangle,
$$
and by \eqref{fraghost} we have
$$\mff_{w}\langle 0,\ldots,0,x_{0},\ldots,x_{n-1}\rangle=\langle 0,\ldots,0,x_{1},\ldots,x_{n-1}\rangle.$$ 
Therefore we have $\alpha_{w,n-1}\circ\phi_w=\mff_{w}\circ \alpha_{w,n}$,
which is the required statement to prove that the diagram at the right
commutes and this completes the proof.
\end{proof}

By Theorem \ref{alpha1} the isomorphism $\alpha$ induces an isomorphism
of prolongation sequences of $R$-algebras
$h_* : J_*(N_{[m]1}A) \map N_{[m]*}A$ over $R_*$.
In particular for all $n \geq 1$, the isomorphism $h_*$ induces an isomorphism
$$
h_{n-1}: J_{n-1}(N_{[m]1}A) \simeq N_{[m]n}A
$$
of $R$-algebras. 

\begin{theorem}
\label{equality}
Let $\mg_*$ be the map of prolongation sequences of $R$-algebras as in 
Theorem \ref{Niso1}. Then $\mg_* = h_*$.
\end{theorem}
\begin{proof}
Note that for $n=1$, $h_0: N_{[m]1}A \map N_{[m]1}A$ is 
the identity map. Hence by the universal property satisfied by the canonical
prolongation sequence as in (\ref{upprol}), $h_*$ is the unique map of 
prolongation sequences induced from the $h_0=\mathbbm{1} \in 
\Hom_R(N_{[m]1}A, N_{[m]1}A)$.

Recall that for $\mg_*:J_*(N_{[m]1}A) \map N_{[m]*}A$, $\mg_1 = \mathbbm{1}$
and hence $\mg_*$ is also the unique map of prolongation sequence induced
from the identity map in 
\newline
$\Hom_R(N_{[m]1}A, N_{[m]1}A)$. 
Hence we have $\mg_* = h_*$ and we are done.
\end{proof}
}


{\bf Acknowledgements.} The authors would like to profusely thank the 
anonymous referee
for giving detailed and perceptive suggestions from which this paper has 
greatly benefitted. They would also like to thank Netan Dogra for helpful 
comments. The second author would like to thank James Borger for 
many insightful discussions. 
He would also like to thank Alessandra Bertapelle
and Nicola Mazzari for stimulating conversations. 
The second author was partially supported by the SERB grant SRG/2020/002248.

\footnotesize{

}

\end{document}